\documentclass[a4paper,12pt,reqno]{amsart}
\usepackage[margin=2.8cm]{geometry}
\usepackage[normalem]{ulem}
\usepackage{faktor}

\usepackage{amssymb,amsmath,bbm,bm}
\usepackage{paralist}
\usepackage{mathtools}
\usepackage{helvet}  
\usepackage{courier} 
\usepackage{type1cm} 
\usepackage{graphicx} 
\usepackage{multicol} 
\usepackage[bottom]{footmisc}
\usepackage{scrextend}
\usepackage[OT1]{fontenc} 
\usepackage[draft]{todonotes} 

\definecolor{darkgreen}{rgb}{0.00,0.33,0.25}
\definecolor{darkred}{rgb}{0.60,0.05,0.05}
\definecolor{darkblue}{rgb}{0.05,0.05,0.60}

\usepackage[colorlinks = true,
 linkcolor = darkgreen,
 anchorcolor = darkgreen,
 citecolor = darkblue,
 filecolor = darkgreen,
 urlcolor = darkgreen
]{hyperref}

\numberwithin{equation}{section}
\allowdisplaybreaks

\DeclareFontFamily{U}{jkpmia}{}
\DeclareFontShape{U}{jkpmia}{m}{it}{<->s*jkpmia}{}
\DeclareFontShape{U}{jkpmia}{bx}{it}{<->s*jkpbmia}{}
\DeclareMathAlphabet{\mathfrak}{U}{jkpmia}{m}{it}
\SetMathAlphabet{\mathfrak}{bold}{U}{jkpmia}{bx}{it}

\newcommand{\e}{\mathrm{e}}

\def\FG{\mathbf}

  \def\bfr{{\FG r}}

\def\BS{\boldsymbol}

 \def\bfmu{{\BS\mu}}
  
  \def\bfrho{{\BS\rho}}

\newcommand{\R}{\mathbb{R}}

\newcommand{\D}{\mathrm{d}}
\newcommand{\dd}{\, \mathrm{d}}
\newcommand{\ddd}{\mathrm{d}}

\newcommand{\sd}{\mathsf{d}}

\newcommand{\sG}{\mathsf{G}}
\newcommand{\sE}{\mathsf{E}}
\newcommand{\sV}{\mathsf{V}}

\newcommand{\eps}{\varepsilon}

\definecolor{jan}{rgb}{0.0,0.3,0.8}
\definecolor{mat}{rgb}{0.0,0.5,0.3}

\newtheorem{theorem}{Theorem}[section]
\newtheorem{proposition}[theorem]{Proposition}
\newtheorem{lemma}[theorem]{Lemma}
\newtheorem{corollary}[theorem]{Corollary}

\theoremstyle{definition}
\newtheorem{definition}[theorem]{Definition}
\newtheorem{example}[theorem]{Example}

\theoremstyle{remark}
\newtheorem{remark}[theorem]{Remark}

\title[Wasserstein gradient flows over metric graphs]{Gradient Flow Formulation of Diffusion equations in the Wasserstein Space over a Metric Graph}

\author{Matthias Erbar}
\author{Dominik Forkert} 
\author{Jan Maas}
\author{Delio Mugnolo}

\address{Matthias Erbar, Fakult\"at f\"ur Mathematik, Universit\"at Bielefeld\\ Postfach 100131, 33501 Bielefeld\\
Germany}
\email{erbar@math.uni-bielefeld.de}

\address{Dominik Forkert,
Institute of Science and Technology Austria (IST Austria)\\
Am Campus 1\\ 
3400 Klos\-ter\-neu\-burg\\ 
Austria}
\email{dominik.forkert@ist.ac.at}

\address{Jan Maas, 
Institute of Science and Technology Austria (IST Austria)\\
Am Campus 1\\ 
3400 Klos\-ter\-neu\-burg\\ 
Austria}

\email{jan.maas@ist.ac.at}
\address{Delio Mugnolo, Lehrgebiet Analysis, Fakult\"at Mathematik und Informatik, Fern\-Universit\"at in Hagen, D-58084 Hagen, Germany}
\email{delio.mugnolo@fernuni-hagen.de}

\def\math#1{$#1$}

\def\reals{{\mathbb R}}

\def\naturalnumbers{{\mathbb N}}

\def\indicator{{\mathbbm 1}}

\def\abs#1{{\left\vert #1 \right\vert}}
\def\norm#1{{\left\Vert #1 \right\Vert}}

\def\calP{{\mathcal P}}
\def\calW{{\mathcal W}}

\def\cT{{\mathcal T}}
\def\cV{{\mathcal V}}
\def\calL{{\mathcal L}}
\def\calI{{\mathcal I}}
\def\calB{{\mathcal B}}

\def\calM{{\mathcal M}}

\def\calF{{\mathcal F}}
\def\calE{{\mathcal E}}
\def\calV{{\mathcal V}}
\def\frakG{{\mathfrak G}}
\def\frakE{{\mathfrak E}}

\def\proj{\operatorname{proj}}
\def\Ent{\operatorname{Ent}}

\def\supp{\operatorname{supp}}

\def\Id{{\operatorname{Id}}}

\def\path{{\operatorname{path}}}
\def\init{{\operatorname{init}}}
\def\term{{\operatorname{term}}}

\def\Lip{{\operatorname{Lip}}}
\def\lip{{\operatorname{lip}}}

\def\argmin{\operatorname*{argmin}}

\def\Span{{\operatorname{span}}}

\def\ext{{\operatorname{ext}}}

\begin{document}

\begin{abstract}
This paper contains two contributions in the study of optimal transport on metric graphs. 
Firstly, we prove a Benamou--Brenier formula for the Wasserstein distance, which establishes the equivalence of static and dynamical optimal transport.
Secondly, in the spirit of Jordan--Kinderlehrer--Otto, we show that McKean--Vlasov
equations can be formulated as gradient flow of the free energy in the Wasserstein space of probability measures. 
The proofs of these results are based on careful regularisation arguments to circumvent some of the difficulties arising in metric graphs, namely, branching of geodesics and the failure of semi-convexity of entropy functionals in the Wasserstein space.
\end{abstract}



\maketitle 

\section{Introduction}
This article deals with the equivalence of static and dynamical optimal transport on metric graphs,
and with a gradient flow formulation of McKean--Vlasov equations in the Wasserstein space of probability measures. 

\smallskip 

Let $(\sV,\sE)$ be a finite (undirected) connected graph and let $\ell: \sE \to (0,\infty)$ be a given weight function.
Loosely speaking, the associated \emph{metric graph} is the geodesic metric space $(\frakG, \sd)$ obtained by identifying the edges $e \in \sE$ with intervals of length $\ell_e$, and gluing the intervals together at the nodes. In other words, a metric graph describes a continuous ``cable system'', rather than a discrete set of nodes; see Section \ref{sec:1} for a more formal definition. 
Metric graphs arise in many applications in chemistry, physics, or engineering, describing quasi-one-dimensional systems such as carbon nano-structures, quantum wires, transport networks, or thin waveguides. They are also widely studied in mathematics; see \cite{berkolaiko2013introduction,Mugnolo} for an overview.

For $1 \leq p < \infty$, 
the $L^p$-Wasserstein distance $W_p$ 
between Borel probability measures $\mu$ and $\nu$ on $\frakG$ 
is defined by the Monge--Kantorovich transport problem
\begin{equation*}
	W_p(\mu, \nu) := 
	\min_{\sigma\in \Pi(\mu, \nu)} \bigg\{ \int_{\frakG \times \frakG} \sd^p(x,y) \dd \sigma(x,y) \bigg\}^{1/p},
\end{equation*}
where $\Pi(\mu,\nu)$ denotes the set of transport plans; i.e., all Borel probability measures $\sigma$ on $\frakG \times \frakG$ with respective marginals $\mu$ and $\nu$.
Since $\frakG$ is compact, the Wasserstein distance $W_p$ metrises the topology of weak convergence of Borel probability measures on $\frakG$, for any $p \geq 1$. 
The resulting metric space of probability measures is the \emph{$L^p$-Wasserstein space over $\frakG$}.

Metric graphs $(\frakG,\sd)$ 
are prototypical examples of metric spaces that exhibit \emph{branching of geodesics}:
it is (typically) possible to find two distinct constant speed geodesics $(\gamma_t)_{t \in [0,1]}$ and $(\tilde\gamma_t)_{t \in [0,1]}$, taking the same values for all times $t \in [0,t_0]$ up to some $t_0 \in (0,1)$.
For this reason, several key results from optimal transport
 are not directly applicable to metric graphs.
This paper contains two of these results.

\smallskip

\subsection*{The Benamou--Brenier formula}

On Euclidean space $\R^d$, a dynamical characterisation of the Wasserstein distance has been obtained in celebrated works of Benamou and Brenier \cite{benamou1999numerical,benamou2000computational}. 
The \emph{Benamou--Brenier formula} asserts that 
\begin{equation}		\label{eq:intro_BB0}
	W_2^2(\mu, \nu) = \inf_{(\mu_t,v_t)} \Bigl\{ \int_0^1 \norm{v_t}_{L^2(\mu_t)}^2 \dd t \Bigr\}
\end{equation}	

where the infimum runs over all 2-absolutely continuous curves $(\mu_t)_{t \in [0,1]}$ in the $L^2$-Wasserstein space over $\R^d$, satisfying the continuity equation
\begin{equation}		
	\label{eq:intro_CE1}
	\frac{\dd}{\dd t} \mu_t + \nabla \cdot (v_t \mu_t) = 0
\end{equation}
with boundary conditions $\mu_0 = \mu$ and $\mu_1 = \nu$.

Here we are interested in obtaining an analogous result in the setting of metric graphs. However, such an extension is not straightforward, since standard proofs in the Euclidean setting (see \cite{AGS,Sant}) make use of the flow map $T : [0,1] \times \R^d \to \R^d$ associated to a
(sufficiently regular) vector field $(v_t)_{t\in[0,1]}$, which satisfies
\begin{equation*}
 \frac{\dd}{\dd t} T(t,x) = v_t\big(T(t,x)\big), \qquad T(0, x) = x, \qquad
 T(t, \cdot)_\# \mu_0 = \mu_t, 
\end{equation*}
see, e.g., \cite[Proposition 8.1.8]{AGS}.
On a metric graph such a flow map $T$ typically fails to exist, since solutions $(\mu_t)_{t \in [0,1]}$ to the continuity equation \eqref{eq:intro_CE1} are usually not uniquely determined by an initial condition $\mu_0$ and a given vector field $(v_t)_{t \in [0,1]}$.

Circumventing this difficulty, Gigli and Han obtained a version of the Benamou--Brenier formula in very general metric measure spaces \cite{gigli2015continuity}.
However, this paper requires a strong assumption on the measures (namely, a uniform bound on the density with respect to the reference measure). 
While this assumption is natural in the general setting of \cite{gigli2015continuity}, it is unnecessarily restrictive in the particular setting of metric graphs.

In this paper, we prove a Benamou--Brenier formula that applies to arbitrary Borel probability measures on metric graphs. 
The key ingredient in the proof is a careful regularisation step for solutions to the continuity equation.

\subsection*{Gradient flow structure of diffusion equations}

As an application of the Bena\-mou--Brenier formula, we prove another central result from optimal transport: the identification of diffusion equations as gradient flow of the free energy in the Wasserstein space $(\calP(\frakG), W_2)$.
In the Euclidean setting, results of this type go back to the seminal work of Jordan, Kinderlehrer, and Otto \cite{jordan1998variational}.

Here we consider diffusion equations of the form
\begin{equation}\label{eq:MKVeqintro}
	\partial_{t}\eta
	= \Delta\eta 	
	+\nabla\cdot\big(\eta\big(\nabla V+\nabla W[\mu]\big)\big),
\end{equation}
for suitable potentials $V : \frakG \to \R$ and $W : \frakG \times \frakG \to \R$.
In analogy with the Euclidean setting, we show that this equation arises as the gradient flow equation 
for a free energy $\calF : \calP(\frakG) \to (-\infty,+\infty]$ composed as the sum of 
entropy, potential, and interaction energies:
\begin{align*}
	\calF(\mu)
	&:=
		\int_{\frakG}
			\eta(x)\log \eta(x)
		\dd\lambda(x)
		+\int_{\frakG} V(x) \eta(x)\dd \lambda(x)
	\\& \qquad	+\frac12
	 	\int_{\frakG \times \frakG} 
		 	W(x,y)\eta(x)\eta(y)
		\dd\lambda(x) \dd\lambda(y),
\end{align*}
if $\mu=\eta\lambda$, where $\lambda$ denotes the Lebesgue measure on $\frakG$.

A key difference compared to the Euclidean setting is that the entropy is \emph{not} semi-convex along $W_2$-geodesics; see Section \ref{sec:2}. 
This prevents us from applying ``off-the-shelf'' results from the theory of metric measure spaces. 
Instead, we present a self-contained proof of the gradient flow result. Its main ingredient is a chain rule for the entropy along absolutely continuous curves, which we prove using a regularisation argument.

Interestingly, we do \emph{not} need to assume continuity of the potential $V$ at the vertices.
Therefore our setting includes diffusion with possibly singular drift at the vertices. 

\medskip 

The Wasserstein distance over metric graphs for $p = 1$ has been studied in \cite{MRT}. The focus is on the approximation of Kantorovich potentials using $p$-Laplacian type problems. 

The recent paper \cite{Burger-Humpert-Pietschmann:2021} deals with dynamical optimal transport metrics on metric graphs. 
The authors start with the dynamical definition \`a la Benamou--Brenier and consider links to several other interesting dynamical transport distances.
The current paper is complementary, as it shows the equivalence of static and dynamical optimal transport, and a gradient flow formulation for diffusion equations. 

Various different research directions involving optimal transport and graphs exist.
 In particular, \emph{dynamical optimal transport on discrete graphs} have been intensively studied in recent years following the papers \cite{Chow-Huang-Li:2012,Maas:2011,Mielke:2011}.
 The underlying state space in these papers is a discrete set of nodes rather than a gluing of one-dimensional intervals.
 
 Another line of research deals with \emph{branched optimal transport}, which is used to model phenomena such as road systems, communication networks, river basins, and blood flow. Here one starts with atomic measures in the continuum, and graphs emerge to describe paths of optimal transport \cite{Xia:2003,Bernot-Caselles-Morel:2009}.

\subsection*{Organisation of the paper}
In Section~\ref{sec:preliminaries} we collect preliminaries on optimal transport and metric graphs.
Section \ref{sec:CE} is devoted to the continuity equation and the Benamou--Brenier formula on metric graphs. 
In particular, we present a careful regularisation procedure for solutions to the continuity equation.
Section~\ref{sec:2} contains an example which demonstrates the lack convexity of the entropy along $W_2$-geodesics in the setting of metric graphs.
Section~\ref{sec:3} deals with the gradient flow formulation of diffusion equations.

\section{Preliminaries}	\label{sec:preliminaries}

In this section we collect some basic definitions and results from optimal transport and metric graphs. 

\subsection{Optimal transport}

In this section we collect some basic facts on the family of \math{L^p}-Wasserstein distances on spaces of probability measures. We refer to \cite[Chapter~5]{Sant}, \cite[Chapter~2]{AG} or \cite[Chapter~6]{Villani} for more details. 

Let \math{(X,d)} be a compact metric space. 
The space of Borel probability measures on \math{X} is denoted by \math{\calP(X)}. 
The pushforward measure $T_\# \mu$ induced by a Borel map $T : X \to Y$ between two Polish spaces is defined by $(T_\# \mu)(A) := \mu(T^{-1}(A))$ for all Borel sets $A \subseteq Y$.

\begin{definition}[Transport plans and maps]\mbox{}
\begin{enumerate}
\item
A \emph{(transport) plan} between probability measures \math{\mu, \nu \in \calP(X)} is a probability measure \math{\sigma \in \calP(X \times X)} with respective marginals \math{\mu} and \math{\nu}, i.e.,
\begin{equation*}
 (\proj_1)_\# \sigma = \mu \qquad \text{and} \qquad (\proj_2)_\# \sigma = \nu,
\end{equation*}
where \math{\proj_i(x_1, x_2 ) := x_i} for \math{i = 1, 2}. 
The set of all transport plans between \math{\mu} and \math{\nu} is denoted by \math{\Pi(\mu, \nu)}.
\item
A transport plan \math{\sigma \in \Pi(\mu, \nu)} is said to be induced by a Borel measurable \emph{transport map} \math{T: X \to X} if \math{\sigma = (\Id, T)_\# \mu}, where $(\Id, T) : X \to X \times X$ denotes the mapping $x \mapsto (x, T(x))$.
\end{enumerate}
\end{definition}

\begin{definition}[Kantorovich--Rubinstein--Wasserstein distance]\label{def:Wass}
For \math{p \geq 1}, the \emph{\math{L^p}-Kantorovich--Rubinstein--Wasserstein distance} between probability measures \math{\mu, \nu \in \calP(X)} is defined by 
\begin{equation}	\label{eq:def:Wasserstein_distance}
W_p(\mu, \nu) := \inf \biggl\{ \biggr( \int_{X \times X} d^p(x,y) \dd \sigma(x,y) \biggr)^{1/p}: \sigma \in \Pi(\mu, \nu) \biggr\}.
\end{equation}
\end{definition}

For any $\mu, \nu \in \calP(X)$ the infimum above is attained by some \math{\sigma_{\min} \in \Pi(\mu, \nu)}; 
any such \math{\sigma_{\min}} is called an \emph{optimal (transport) plan} between \math{\mu} and \math{\nu}. 
If a transport map \math{T} induces an optimal transport plan, we call \math{T} optimal as well.

By compactness of $(X,d)$, the \math{L^p}-Wasserstein distance metrises the weak convergence in \math{\calP(X)} for any \math{p \geq 1}; see, e.g., \cite[Corollary 6.13]{Villani}. Moreover, \math{(\calP(X), W_p)} is a compact metric space as well.

We conclude this section with a dual formula for the Wasserstein distance (see, e.g., \cite[Section~1.6.2]{Sant}).
To this aim, we recall that for 
\math{c : X \times X \to \R}, 
the \emph{\math{c}-transform} of a function \math{\varphi: X \to \reals \cup \{+\infty\}} is defined by \math{\varphi^c(y) := \inf_{x \in X} c(x,y) - \varphi(x)}.
A function \math{\psi: X \to \reals \cup \{+\infty\}} is called \emph{\math{c}-concave} if there exists a function \math{\varphi: X \to \reals \cup \{+\infty\}} such that \math{\psi = \varphi^c}.

\begin{proposition}[Kantorovich duality]
\label{prop:Kantorovich-duality}
For any \math{\mu, \nu \in \calP(X)} we have
\begin{equation*}
	W_p^p(\mu, \nu)
		 = \sup_{\varphi, \psi \in C(X)}
		 	 \biggl\{ \int_X \varphi \dd \mu 
		 		 	 + \int_X \psi \dd \nu 
					 	\ : \ 
			\varphi(x) + \psi(y) \leq d^p(x,y) 
						 \quad \forall x,y \in X \biggr\}.
\end{equation*}
Moreover, the supremum is attained by a maximising pair of the form \math{(\varphi, \psi) = (\varphi, \varphi^c)}, where \math{\varphi} is a \math{c}-concave function, for $c(x,y):= d^p(x,y)$.
\end{proposition}

Any maximiser $\varphi$ is called a \emph{Kantorovich potential}.

\subsection{Metric graphs}
\label{sec:1}

In this subsection we state some basic facts on metric graphs; see e.g., \cite{MRT}, \cite{berkolaiko2013introduction} or \cite{kuchment2008quantum} for more details.

\begin{definition}[Metric graph]
Let \math{\sG=(\sV,\sE, \ell)} be an oriented, weighted graph, which is finite and connected.
We identify each edge \math{e=(e_\init, e_\term) \in \sE} with an interval \math{(0,\ell_e)} and the corresponding nodes \math{e_\init, e_\term \in \sV} with the endpoints of the interval, ($0$ and $\ell_e$ respectively). 
The \emph{spaces of open and closed metric edges} over \math{\sG} are defined as the respective topological disjoint unions
\begin{equation*}
\frakE := \coprod_{e \in \sE} (0, \ell_e) \qquad \text{and} \qquad \overline{\frakE} := \coprod_{e \in \sE} [0, \ell_e].
\end{equation*}
The \emph{metric graph over \math{\sG}} is the topological quotient space
\begin{equation*}
\frakG := \faktor{\overline{\frakE} }{\sim},
\end{equation*}
where points in $\overline{\frakE}$ corresponding to the same vertex are identified. 
\end{definition}

Note that the orientation of \math{e} determines the parametrisation of the edges, but does not otherwise play a role.
To distinguish ingoing and outgoing edges at a given node, 
we introduce the signed incidence matrix $\mathcal I=(\iota_{ev})$ whose entries are given by
\begin{align*}
\iota_{ev}:=\begin{cases}
 +1 \quad&\hbox{if }v=e_\init,\\
 -1 &\hbox{if }v=e_\term,\\
0 &\hbox{otherwise}.
\end{cases}
\end{align*}

As a quotient space, any metric graph naturally inherits the structure of a metric space from the Euclidean distance on its metric edges \cite[Chapter 3]{Burago-Burago-Ivanov:2001}: indeed, under our standing assumption that $\frakG$ is connected the quotient semi-metric $\sd$ becomes a metric.

The distance $\sd:\frakG\times\frakG\to[0,\infty)$ on $\frakG$ can be more explicitly described as follows:
For \math{x,y \in \frakG}, let $\tilde \sG = (\tilde \sV,\tilde\sE)$ be the underlying discrete graph obtained by 
adding new vertices at \math{x} and \math{y}, and let $\sd(x,y)$ be the weighted graph distance between $x$ and $y$ in $\tilde \sG$, i.e.,
\[\sd(x,y):=\min \sum_{i=1}^n \ell_{e_i},  \]
where the minimum is taken over all sequences of vertices $x=x_0,\dots, x_n=y$ in the extended graph $\tilde \sG$ such that $x_{i-1}$ and $x_i$ are vertices of an edge  $e_i\in \tilde \sE$ for all $i=1,\dots,n$.
In particular, if $x$ and $y$ are vertices in $\sG$, no new nodes are added, and we recover the graph distance in the original graph. We refer to~\cite[Chapter~3]{Mugnolo} for details.

By construction, the distance function \math{\sd} metrises the topology of \math{\frakG}. It is readily checked that $\sd$ is a geodesic distance, i.e., each pair of points $x,y\in\frakG$ can be joined by a curve of minimal length $\sd(x,y)$. Consequently, also the Wasserstein space $(\calP(\frakG),W_2)$ is a geodesic space.

\medskip
In a metric space $(X,d)$, recall that the \emph{local Lipschitz constant} of a function \math{f: X \to \reals} is defined by
\begin{equation*}
	\lip(f)(x) := \limsup_{y \to x} \frac{\abs{f(y) - f(x)}}{d(x,y)},
\end{equation*}
whenever \math{x \in X} is not isolated, and \math{0} otherwise. 
The \emph{(global) Lipschitz constant} is defined by 
\begin{equation*}
	\Lip(f) := \sup_{y \neq x} \frac{\abs{f(y) - f(x)}}{d(x,y)}.
\end{equation*}
If the underlying space $X$ is a geodesic space, we have \math{\Lip(f) = \sup_x \lip(f)(x)}.

\medskip

At the risk of being redundant, we explicitly introduce a few relevant function spaces, although they are actually already fully determined by the metric measure structure of the metric graph $\frakG$.

\begin{enumerate}[(i)]
	\item 
\math{C(\frakG)} denotes the space of continuous real-valued functions on \math{\frakG}, endowed with the uniform norm \math{\norm{\cdot}_\infty}. 
\item 
\math{C^k(\overline \frakE)} is the space of all functions $\varphi$ on \math{\overline \frakE} such that the restriction to each closed edge has continuous derivatives up to order \math{k \in \naturalnumbers}.
\item 
\math{L^p(\frakG)}, for \math{p \in [1,\infty]}, is the \math{p}-Lebesgue space over the measure space \math{(\frakG, \lambda)}, where \math{\lambda} denotes the image of the 1-dimensional Lebesgue measure on $\overline\frakE$ under the quotient map.
\item Likewise, we consider the Sobolev spaces \math{W^{1,p}(\overline \frakE)} of $L^p(\frakG)$-functions whose restriction on each edge is weakly differentiable with weak derivatives in $L^p(\frakG)$.
\end{enumerate}

\section{The continuity equation on a metric graph}		\label{sec:CE}

In this section we fix a metric graph $\frakG$ and perform a study of the continuity equation
\begin{align} 
	\label{eq:cont-eq}
	\partial_t \mu_t + \nabla \cdot J_t = 0
\end{align}
in this context.

\subsection{The continuity equation} 

In this work we mainly deal with \emph{weak} solutions to the continuity equation, which will be introduced in Definition \ref{def:dist_et_weak_solution1}. 
To motivate this definition, we first introduce the following notion of strong solution.

\begin{definition}[Strong solutions to the continuity equation]\label{def:CE-strong}
A pair of measurable functions $(\rho, U)$ with $\rho: (0,T) \times \frakG \to \R_+$ and $U : (0,T) \times \overline{\frakE} \to \R$ is said to be a strong solution to \eqref{eq:cont-eq} if
\begin{enumerate}[$(i)$]
\item \label{it:rho-reg}
$t \mapsto \rho(t,x)$ is continuously differentiable for every $x \in \frakG$;
\item \label{it:U-reg} $x \mapsto U_t(x)$ belongs to $C^1(\overline \frakE)$ for every $t \in (0,T)$;
\item the continuity equation $\frac{\D}{\D t} \rho_t(x) + \nabla \cdot U_t(x) = 0$ holds for every $t \in (0,T)$ and $x \in \frakE$;
\item \label{it:node} for every $t \in (0,T)$ and $w \in \sV$ we have
$\sum_{e \in \sE_w} \iota_{ew} U_t(w_e) = 0$.
\end{enumerate}
Here, we write $\rho_t := \rho(t, \cdot)$ and $U_t := U(t, \cdot)$ and denote by $\nabla$ the spatial derivative.
Moreover, $\sE_w$ denotes the set of all edges adjacent to the node $w \in \sV$, and $w_e \in \overline{\frakE}$ denotes the corresponding endpoint of the metric edge $e$ which corresponds to $w \in \frakG$.
\end{definition}

To motivate the definition of a weak solution, suppose that we have a strong solution $(\rho_t, U_t)_{t \in (0,T)}$ to the continuity equation \eqref{eq:cont-eq}. 
Let 
$\psi \in C^1(\overline\frakE) 
\cap
C({\frakG})$ 
be a test function.
Integration by parts on every metric edge $e$ gives
\begin{align*}
\frac{\D}{\D t} \int_0^{\ell_e} \psi \rho_t \dd x
	= \int_0^{\ell_e} 
		\nabla \psi \cdot U_t \dd x 
	+ \psi U_t \Big|_{0}^{\ell_e},
\end{align*}
and summation over $e \in \sE$ yields
\begin{align*}
	\frac{\D}{\D t} \int_{\frakG} \psi \rho_t \dd x
	= \int_{\overline \frakE} 
		\nabla \psi \cdot U_t \dd x
	 + \sum_{w \in \sV} \psi(w)
	 \sum_{e \in \sE_w} 
		\iota_{ew} U_t(w_e)
	= \int_{\overline \frakE} 
		\nabla \psi \cdot U_t \dd x,
\end{align*}
where we use the continuity of $\psi$ on $\frakG$ as well as the node condition \eqref{it:node} above in the last step. This ensures that the net ingoing momentum vanishes at every node in $\sV$. In particular, choosing $\psi\equiv1$ yields
\begin{equation*}
 \int_\frakG\rho_s\dd \lambda = \int_\frakG\rho_t\dd \lambda,
\end{equation*}
for all $s, t \in (0,T)$, i.e.\ solutions to the continuity equation are mass-preserving. Here
condition \eqref{it:node} is crucial to ensure that no creation or annihilation of mass occurs at the nodes.

\begin{definition}[Weak solution]\label{def:dist_et_weak_solution1}
A pair $(\mu_t, J_t)_{t \in (0,T)}$ consisting of probability measures \math{\mu_t} on \math{\frakG} and signed measures \math{J_t} on $\overline\frakE$,
such that  $t \mapsto \big(\mu_t(A), J_t(A)\big)$ is measurable for all Borel sets $A$,
is said to be a weak solution to \eqref{eq:cont-eq} if
\begin{enumerate}[$(i)$]
\item \math{t \mapsto \int_{\frakG} \psi \dd \mu_t} is absolutely continuous for every \math{\psi \in 
C^1(\overline\frakE) 
\cap
C({\frakG})};

\item $\int_0^T|J_t|(\overline\frakE)\dd t<\infty$;

\item for every $\psi \in 
C^1(\overline\frakE) 
\cap
C({\frakG})$ and a.e.\ $t \in (0,T)$, we have
\begin{equation}
\label{eq:ce_weak1}
\frac{\D}{\D t} \int_{\frakG} \psi \dd \mu_t = \int_{\overline\frakE} \nabla \psi \cdot \D J_t.
\end{equation}
\end{enumerate}
\end{definition}

\begin{remark}\label{rem:distributional}
Proposition \ref{prop:function-reg} below shows that continuous functions on $\frakG$ can be uniformly approximated by $C^1$ functions.
By standard arguments it then follows that $(\mu_t,J_t)_{t\in(0,T)}$ is a weak solution if and only if the following conditions hold: {$(i')$} $t\mapsto \mu_t$ is weakly continuous; {$(ii)$} from Definition \ref{def:dist_et_weak_solution1} holds; and 
\begin{enumerate}
\item[$(iii')$] for every 
$\varphi \in C^1_c\big((0,T) \times\overline\frakE\big)$
such that  
$\varphi_t,\partial_t\varphi_t
\in 
	C^1(\overline\frakE) \cap C({\frakG})$
for all $t \in (0,T)$	
we have
\begin{equation}\label{eq:dist_et_weak_solution2}
	\int_0^T\! \biggl( \int_{\frakG} 
				\partial_t \varphi \dd\mu_t + 
	\int_{\overline\frakE} \nabla \varphi\cdot \D J_t \biggr) \D t = 0.
\end{equation}
\end{enumerate}
See Lemma \ref{lem:regularised_CE1_weakly_continuous} below for the weak continuity in $(i')$.
\end{remark}

\bigskip
The next result asserts that the momentum field does not give mass to vertices for a.e.\ time point. Hence, we can equivalently restrict the integrals over $\overline{\frakE}$ in \eqref{eq:ce_weak1} and \eqref{eq:dist_et_weak_solution2} to the space of open edges $\frakE$.

\begin{lemma}\label{lem:no-momentum-vertices}
Let $B := \overline\frakE \setminus \frakE$ denote the set of all boundary points of edges. For any weak solution to the continuity equation $(\mu_t, J_t)_{t \in (0,T)}$, we have 
 \begin{equation*}
 \int_0^T|J_t|(B) \dd t = 0.
 \end{equation*}
\end{lemma}

\begin{proof}
Fix a metric edge $e$ in $\overline\frakE$ and take $w \in \{e_{\rm init}, e_{\rm term}\}$. 
Without loss of generality we take 
	$w=e_{\rm init}$. 
Then we can construct a family of functions 
	$\varphi_\eps$ for $\eps\in(0,\ell_e)$
with the following properties: 
\begin{enumerate}[$(a)$]
	\item $\varphi_\eps\in C^1(\overline\frakE) \cap C(\frakG)$; 
	\item $\varphi_\eps \equiv 0$ on $\overline\frakE\setminus e$ and
	$\varphi_\eps\to 0$ uniformly on $\frakG$ as $\eps\to0$; 
	\item $|\nabla\varphi_\eps|=1$ on $(0,\eps)\subset e$ 
	and
	$|\nabla\varphi_\eps|\to0$ uniformly on compact subsets of $\overline\frakE\setminus\{w\}$. 
\end{enumerate}
For instance, we could set $\varphi^\varepsilon(x) := \indicator_{e}(x) \eta_\eps(x)$, where $\eta_\eps:\R\to\R$ is a $C^1$ approximation of the function $x\mapsto \big(x\wedge \frac{\eps(\ell_e-x)}{\ell_e-\eps}\big)\vee0$. 
Choosing \math{\varphi(t,x) = \varphi^\varepsilon(x)} in \eqref{eq:dist_et_weak_solution2}, we obtain by passing to the limit $\eps \to 0$ that $\int_0^T |J_t| (\{w\})\D t = 0$.
\end{proof}

\begin{lemma}[Weak and strong solutions]\label{lem:Lip} 
The following assertions hold:
\begin{enumerate}[(i)]
\item If $(\rho_t , U_t)_{t \in (0,T)}$ is a strong solution to the continuity equation, then the pair $(\mu_t , J_t)_{t \in (0,T)}$ defined by $\mu_t = \rho_t \lambda$ and $J_t = U_t \lambda$ is a weak solution to the continuity equation.
\item If $(\mu_t , J_t)_{t \in (0,T)}$ is a weak solution to the  continuity equation \eqref{eq:ce_weak1} such that the densities $\rho_t := \frac{\ddd \mu_t}{\ddd \lambda}$ and $U_t := \frac{\ddd J_t}{\ddd \lambda}$ exist for all times $t \in (0,T)$ and satisfy the regularity conditions \eqref{it:rho-reg} and \eqref{it:U-reg} of Definition \ref{def:CE-strong}, then $(\rho_t, U_t)_{t \in (0,T)}$ is a strong solution to the continuity equation.
\end{enumerate}
\end{lemma}

\begin{proof}
Both claims are straightforward consequences of integration by parts on each metric edge in $\frakE$. 
\end{proof}

\subsection{Characterisation of absolutely continuous curves}

Let \math{(X,d)} be a metric space 
and let $T > 0$.

\begin{definition}
For $p \geq 1$, we say that a curve \math{\gamma: (0,T) \to X} is \emph{$p$-absolutely continuous} if there exists a function \math{g \in L^p(0,T)} such that
\begin{equation}\label{eq:AC_def1}
d(\gamma_s, \gamma_t) \leq \int_s^t g_r \dd r \qquad \forall s,t \in (0,T): s \leq t.
\end{equation}
\end{definition}

The class of $p$-absolutely continuous curves is denoted by $AC^p\big((a,b);(X,d)\big)$. 
For $p=1$ we simply drop $p$ in the notation. 
The notion of \emph{locally}
$p$-absolutely continuous curve is defined analogously.

\begin{proposition}
Let $p \geq 1$. 
For every $p$-absolutely continuous curve \math{\gamma: (0,T) \to X}, the \emph{metric derivative} defined by
\begin{equation*}
	\abs{\dot \gamma}(t) 
		:= 
	\lim_{s \to t}	
	 \frac{d(\gamma_s, \gamma_t)}{\abs{s-t}}
\end{equation*}
exists for a.e.\ \math{t \in (0,T)} and $t\mapsto 	\abs{\dot \gamma}(t) $ belongs to \math{L^p(0,T)}. 
The metric derivative \math{\abs{\dot \gamma}(t)} is an admissible integrand in the right-hand side of \eqref{eq:AC_def1}.
Moreover, any other admissible integrand \math{g \in L^p(0,T)} satisfies
$\abs{\dot \gamma}(t) \leq g(t)$ for a.e.\ $t \in (0,T)$.
\end{proposition}

\begin{proof}
See, e.g., \cite[Theorem~1.1.2]{AGS}. 
\end{proof}

The next result relates the metric derivative of \math{t \mapsto \mu_t} to the \math{L^2(\mu_t)}-norm of the corresponding vector fields \math{v_t}.

\begin{theorem}[Absolutely continuity curves]
\label{thm:continuity1}
The following statements hold:
\begin{enumerate}[(i)]
\item \label{it:continuity1i}
 If \math{(\mu_t)_{t \in (0,T)}} is absolutely continuous in \math{(\calP(\frakG),W_2)}, then there exists, for a.e.\ \math{t \in (0,T)}, a vector field \math{v_t \in L^2(\mu_t)} such that \math{\Vert v_t \Vert_{L^2(\mu_t)} \leq \vert \dot \mu \vert(t)} and $(\mu_t, v_t \mu_t)_{t \in (0,T)}$ is a weak solution to the continuity equation \eqref{eq:ce_weak1}.

\item \label{it:continuity1ii}
Conversely, if \math{(\mu_t, v_t \mu_t)_{t \in (0,T)}} is a weak solution to the continuity equation \eqref{eq:ce_weak1} satisfying \math{\int_0^1 \Vert v_t \Vert_{L^2(\mu_t)} \dd t < + \infty}, then \math{(\mu_t)_{t \in (0,T)}} is an absolutely continuous curve in \math{(\calP(\frakG), W_2)} and $\vert \dot \mu \vert(t) \leq \Vert v_t \Vert_{L^2(\mu_t)}$ for a.e.\ $t \in (0,T)$. 
\end{enumerate}
\end{theorem}

\begin{proof}[Proof of \eqref{it:continuity1i}]
We adapt the proof of \cite[Theorem 8.3.1]{AGS} to the setting of metric graphs.

The idea of the proof is as follows: 
On the space-time domain $Q := (0,T) \times \frakG$ we consider the Borel measure $\bfmu := \int_0^T \delta_t \otimes \mu_t \dd t$ whose disintegration with respect to the Lebesgue measure on \math{(0,T)} is given by $(\mu_t)_{t \in (0,T)}$. 
To deal with the fact that gradients of smooth functions are multi-valued at the nodes, we define $\overline\mu_t \in \calM_+(\overline \frakE)$ by $\overline\mu_t(A) := \sum_{e \in \sE} \mu_t(A \cap \overline e)$ for every Borel set $A \subseteq \overline\frakE$.
We then set $\overline Q := (0,T) \times \overline\frakE$ and define $\overline \bfmu \in \calM_+(\overline Q)$ by $\overline \bfmu := \int_0^T \delta_t \otimes \overline\mu_t \dd t$. (Note that mass at the nodes is counted multiple times).
Consider the linear spaces of functions $\cT$ and $\cV$ given by 
\begin{align*}
	\cT & := \Span 
		\bigg\{ 
		(0,T) \times \frakG \ni (t,x) 
			\mapsto
		a(t) \varphi(x) 
			\ : \ 
		a \in C_c^1(0,T), \, \varphi \in 
		C^1(\overline \frakE) \cap C(\frakG)
		\bigg\}, \\
	\cV & := \bigg\{ 
		(0,T) \times \overline\frakE \ni (t,x) 
			\mapsto
				\nabla_x \Phi(t,x) 
			\ : \ 
		\Phi \in \cT
		\bigg\}.
\end{align*}
The strategy is to show that the linear functional \math{L: \cV \to \reals} given by
\begin{equation*}
	L(a \otimes \nabla \varphi) 
		:= - \int_Q \dot a(t) \varphi(x) 
		\dd \bfmu(x,t), 
 \end{equation*}
is well-defined and $L^2(\overline Q, \overline\bfmu)$-bounded with $\norm{L}^2 \leq \int_0^T \abs{\dot \mu}^2(t) \dd t$. 
Once this is proved, 
the Riesz Representation Theorem yields the existence of a vector field $\pmb v$ in $\overline{\cV} \subseteq L^2(\overline Q, \overline \bfmu)$ such that $\|\pmb v\|^2_{L^2(\overline \bfmu)} \leq \int_0^T \abs{\dot \mu}^2(t)\dd t$ and
\begin{equation}\label{eq:metric_est3}
	 -\int_0^T \dot a(t) \int_\frakG \varphi(x) \D \mu_t(x) \dd t 
	= L(a \otimes \nabla \varphi)
	= \int_0^T a(t) \int_{\overline{\frakE}} \nabla \varphi(x) v_t(x)
			\dd \overline{\mu}_t(x) \dd t
\end{equation}
for $v_t := \pmb v(t, \cdot)$ and all $a \in C_c^1(0,T)$ and $\varphi \in C^1(\overline \frakE) \cap C(\frakG)$.

Once this is done, we show that the momentum vector field $\pmb J:= \pmb v \cdot \bfmu$ does not assign mass to boundary points in $\overline\frakE$, so that $\pmb v$ can be interpreted as an element in $L^2(Q,\bfmu)$ and the integration over vector fields can be restricted to $\frakE$.

\medskip

\emph{Step 1.} \ 
Fix a test function $\varphi \in C^1(\overline \frakE) \cap C(\frakG)$
and consider the bounded and upper semicontinuous function $H : \frakG \times \frakG \to \R$ given by
\begin{equation*}
H(x,y):= \begin{cases} \displaystyle
		\lip(\varphi)(x) &\text{if } x=y, \\
		\displaystyle \frac{\abs{\varphi(x) - \varphi(y)}}{\sd(x,y)} &\text{if } x \neq y,
\end{cases}
\end{equation*}
for \math{x,y \in \frakG}.
For $s, t \in (0,T)$, let $\sigma^{s \to t} \in \Pi(\mu_s, \mu_t)$ be an optimal plan. The Cauchy--Schwarz inequality yields
\begin{equation}\begin{aligned}
\label{eq:metric_est0}
	\bigg| \int_\frakG \varphi \dd \mu_s 
		 - \int_\frakG \varphi \dd \mu_t \bigg|
	& \leq \int_{\frakG \times \frakG} 
			\sd(x,y) H(x,y) \dd \sigma^{s \to t}(x,y) 
 \\	& \leq W_2(\mu_s, \mu_t) 
 			\bigg( 
			\int_{\frakG \times \frakG} H^2(x,y)
						 \dd \sigma^{s \to t}(x,y)
			\bigg)^{1/2}.
\end{aligned}\end{equation}
As $\varphi$ is globally Lipschitz on \math{\frakG}, we obtain
\begin{equation*}
\Bigl\vert \int_\frakG \varphi \dd \mu_s - \int_\frakG \varphi \dd \mu_t \Bigr\vert \leq \Lip(\varphi) W_2(\mu_s,\mu_t)
\end{equation*}
and infer that the mapping $t \mapsto \int_\frakG \varphi \dd \mu_t$ is absolutely continuous, hence, differentiable outside of a null set $N_\varphi \subseteq (0,T)$.

Fix $t \in (0,T)$ and take a sequence \math{\{s_n\}_{n \in \naturalnumbers}} converging to \math{t}. 
Since $\{\mu_{s_n}\}$ is weakly convergent, this sequence is tight. Consequently, $\{\sigma^{s_n \to t}\}_{n \in \naturalnumbers}$ is tight as well, and we may extract a subsequence converging weakly to some $\hat\sigma \in \calP(\frakG \times \frakG)$. It readily follows that $\hat\sigma \in \Pi(\mu_t, \mu_t)$. Moreover, along the convergent subsequence, we have 
\begin{align*}
	\int_{\frakG \times \frakG} \sd^2(x,y) \dd \hat\sigma(x,y)
	\leq \liminf_{n \to \infty}
	\int_{\frakG \times \frakG} \sd^2(x,y) \dd \sigma^{s_n \to t}(x,y)	
	 = \liminf_{n \to \infty} W_2^2(\mu_{s_n}, \mu_t)
	 = 0,
\end{align*}
which implies that $\hat\sigma = (\Id, \Id)_\# \mu_t$.

Using this result and the upper-semicontinuity of $H$, it follows from \eqref{eq:metric_est0} that
\begin{equation}\begin{aligned}\label{eq:metric_est1}
	\limsup_{s \to t} \bigg|
		\frac{ 
			 \int_\frakG \varphi \dd \mu_s 
			- \int_\frakG \varphi \dd \mu_t 
			 }{s-t} \bigg|
	& \leq 
		\abs{\dot \mu}(t) 
			\limsup_{s \to t}
			\bigg( 
				\int_{\frakG \times \frakG} H^2(x,y)
						 \dd \sigma^{s \to t}(x,y)
			\bigg)^{1/2}	
\\ & \leq \abs{\dot \mu}(t) 
			\cdot \norm{\lip(\varphi)}_{L^2(\mu_t)}.
\end{aligned}\end{equation}

\medskip
\emph{Step 2.} \ 
Take $\Phi \in \cT$. 
Using dominated convergence, Fatou's Lemma, and \eqref{eq:metric_est1}, we obtain
\begin{equation}\begin{aligned}
\label{eq:metric_est2}
	\bigg| \int_Q \frac{\D}{\D t} \Phi(x,t) \dd \bfmu(x,t) \bigg|
	& = \lim_{h \searrow 0} 
		\bigg| \frac{1}{h} 
			\int_Q \Phi(x, t-h) - \Phi(x,t) \dd \bfmu(x,t)
	 	\bigg|
 \\ & = \lim_{h \searrow 0} 
 		\bigg| \frac{1}{h} \int_0^T
			\bigg(
				\int_\frakG \Phi(x,t) \dd \mu_{t+h}(x) 
				 - \int_\frakG \Phi(x,t) \dd \mu_t(x) 
			\bigg) \dd t
		\bigg|
 \\ & \leq \int_0^T \abs{\dot \mu}(t)
 		 \cdot \norm{\lip_x(\Phi)(\cdot,t)}_{L^2(\mu_t)} \dd t 
 \\ &	\leq \bigg( 
				\int_0^T \abs{\dot \mu}^2(t) \dd t 
		 \bigg)^{1/2} 
		 \bigg( 
		 		\int_Q \abs{\lip_x(\Phi)(x,t)}^2 \dd \bfmu(x,t) 	
		 \Bigr)^{1/2}.
\end{aligned}\end{equation}
Since $\int_Q |\lip_x(\Phi)(x,t)|^2 \dd \bfmu(x,t) \leq \int_{\overline Q} |\nabla \Phi (x,t)|^2 \dd \overline\bfmu(x,t)$, we infer that $L$ is well-defined and extends to a bounded linear functional on the closure of $\cV$ in $L^2(\overline Q, \overline \bfmu)$ with $\norm{L}^2 \leq \int_0^T \abs{\dot \mu}^2(t) \dd t$, 
which allows us to apply the Riesz Representation Theorem, as announced above.

In particular, \eqref{eq:metric_est3} implies that \math{t \mapsto \int_{\frakE} \nabla \varphi \cdot v_t \dd \mu_t} is a distributional derivative for \math{t \mapsto \int_\frakG \varphi \dd \mu_t}. Since the latter function is absolutely continuous and, therefore, belongs to the Sobolev space \math{W^{1,1}(0,T)}, we obtain 
\begin{equation}\label{eq:metric_est4}
\frac{\D}{\D t} \int_\frakG \varphi \dd \mu_t = \int_{\overline\frakE} \nabla \varphi \cdot v_t \dd \mu_t \qquad \text{for a.e.\ } t \in (0,T).
\end{equation}
We conclude that \math{(\mu_t, v_t)_{t \in (0,T)}} solves the continuity equation in the weak sense. Lemma \ref{lem:no-momentum-vertices} implies that for a.e.\ $t$ the momentum field \math{J_t := v_t \cdot \overline\mu_t} does not give mass to any boundary point in \math{\overline{\frakE}} .
Consequently, the spatial domain of integration on the right-hand side of \eqref{eq:metric_est3} may be restricted to \math{\frakE}.
\medskip

\emph{Step 4.} \ 
It remains to verify (by a standard argument) the inequality relating the \math{L^2(\mu_t)}-norm of the vector field \math{v_t} to the metric derivative of \math{\mu_t}. 

For this purpose, fix a sequence \math{(\pmb \varpi_i)_{i \in \naturalnumbers}} of functions \math{\pmb \varpi_i \in \cV} converging to \math{\pmb v} in \math{L^2(\overline{\bfmu})} as \math{i \to \infty}. 
For every compact interval \math{I \subseteq (0,T)} and \math{a \in C^1(0,T)} satisfying \math{0 \leq a \leq 1} and \math{\supp a = I}, we then obtain
\begin{align*}
\int_{Q} a(t) \abs{\pmb v(x,t)}^2 \dd \bfmu(x,t) 
		& = \lim_{i \to \infty}
				 \int_{Q} a(t) \pmb \varpi_i(x,t) \pmb v(x,t) \dd \bfmu(x,t) 
	 \\	& = \lim_{i \to \infty} L(a \pmb \varpi_i ) 
	 \leq 	\Bigl(\int_0^T \indicator_I \abs{\dot \mu_t}^2 \dd t \Bigr)^{1/2} 		 	\lim_{i \to \infty} 
	 		\Bigl( \int_{Q} \indicator_I 
					\vert \pmb \varpi_i \vert^2 \dd \bfmu \Bigr)^{1/2} 
	\\ & = \Bigl(\int_0^T \indicator_I 
			\abs{\dot \mu}^2(t) \dd t \Bigr)^{1/2} 
			\Bigl( \int_{Q} \indicator_I \abs{v}^2 \dd \bfmu \Bigr)^{1/2}.
\end{align*}
Letting $\norm{a - \indicator_I}_\infty \to 0$, this inequality implies
\begin{equation*}
	\int_I \int_\frakE \abs{v_t}^2 \dd \mu_t \dd t
		 \leq \int_I \abs{\dot \mu}^2(t) \dd t.
\end{equation*}
Since \math{I \subseteq (0,T)} is arbitrary, this implies that \math{\norm{v_t}_{L^2(\mu_t)} \leq \abs{\dot \mu}(t)} for a.e.\ \math{t \in (0,T)}. 
\end{proof}

\subsection{Regularisation of solutions to the continuity equation}\label{subsec:reg_CE}

Next we introduce a suitable spatial regularisation procedure for solutions to the continuity equation. This will be crucial in the proof of the second part of Theorem \ref{thm:continuity1}. 

Let $\eps > 0$ be sufficiently small, i.e.,
	 such that \math{2\varepsilon} is strictly smaller than the length of every edge in \math{\sE}.
We then consider the supergraph \math{\frakG_\ext \supseteq \frakG} defined by adjoining an auxiliary edge \math{e^\ext_v} of length \math{2 \varepsilon} to each node \math{v \in \sV} (see Figure~\ref{fig:reg_curve1}).
The corresponding set of metric edges will be denoted by \math{\frakE_\ext \supset \frakE}.

We next define a regularisation procedure for functions based on averaging.
The crucial feature here is that non-centred averages are used, to ensure that the regularised function is continuous.

In the definition below, 
we parametrise each edge \math{e=(e_\init, e_\term) \in \sE} using the interval \math{(-\frac{\ell_e}{2}, \frac{\ell_e}{2})} instead of \math{(0,\ell_e)}. 
The auxiliary edges \math{e^\ext_{e_\init}} and \math{e^\ext_{e_\term}} will then be identified with the intervals \math{(-\frac{\ell_e}{2} - 2\varepsilon, - \frac{\ell_e}{2})} and \math{( \frac{\ell_e}{2}, \frac{\ell_e}{2} + 2 \varepsilon)}, respectively.
We stress that for each vertex $v$, there is only one additional edge, but we use several different parametrisations for it -- one for each edge incident in $v$.

\begin{figure}
	\includegraphics[width=0.28\textwidth]{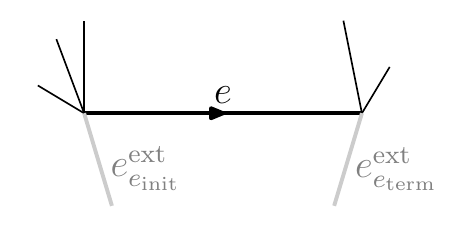}
	\caption{The supergraph \math{\frakG_\ext} is constructed by adjoining an additional leaf at every node in \math{V}.}
	\label{fig:reg_curve1}
\end{figure}

\begin{definition}[Regularisation of functions]\label{def:regularised_curves1}
For \math{\varphi \in L^1(\frakG_\ext)}, we define \math{\varphi^\varepsilon: \frakG \to \reals} by
\begin{equation}\label{eq:kernel_approx1_function1}
\varphi^\varepsilon(x) := \frac{1}{2 \varepsilon} 
		\int_{\alpha^\varepsilon_e x - \varepsilon}
		 ^{\alpha^\varepsilon_e x+ \varepsilon} 
		 	\varphi(y) \dd y 
	\qquad \text{for } x \in e=[-\tfrac{\ell_e}{2}, \tfrac{\ell_e}{2}],
\end{equation}
where \math{\alpha^\varepsilon_e := (\ell_e+2 \varepsilon)/\ell_e}. 
We write $\alpha^\varepsilon(x) := \alpha^\varepsilon_e$, whenever $x \in \frakE$ is a point on the metric edge $e$.
\end{definition}

Note that the value of $\varphi_\varepsilon$ in each of the nodes depends only on data on the corresponding auxiliary edge. In particular, the value at the nodes does not depend on the choice of the edge, so that $\varphi_\varepsilon$ indeed defines a function on $\frakG$. 
We collect some basic properties of this regularisation in the following result.

\begin{proposition}\label{prop:function-reg}
The following properties hold for every $\eps > 0$  sufficiently small:
\begin{enumerate}[(i)]
\item \label{it:fct-reg}
 Regularising effect: For any \math{\varphi \in C(\frakG_\ext)} we have $\varphi^\varepsilon 
 \in 
 C^1(\overline\frakE) 
 \cap
C({\frakG})$ and
\begin{equation}\label{eq:regularised_diff_quotient1}
\nabla \varphi^\varepsilon(y) = \frac{\alpha^\varepsilon_e}{2 \varepsilon} \bigl( \varphi(\alpha^\varepsilon_e y+ \varepsilon) - \varphi(\alpha^\varepsilon_e y- \varepsilon) \bigr) 
\end{equation}
for $y \in [-\ell_e/2, \ell_e/2]$.
\item \label{it:fct-conv}
 If \math{\varphi} belongs to \math{C(\frakG_\ext)}, then \math{\varphi^\varepsilon} converges uniformly to \math{\varphi \vert_\frakG} as \math{\varepsilon \searrow 0}. 
\end{enumerate}
\end{proposition}

\begin{proof}
\eqref{it:fct-reg} follows by direct computation; \eqref{it:fct-conv} follows using the uniform continuity of $\varphi$ on $\frakG_\ext$.
\end{proof}

\begin{lemma}[Weak continuity]\label{lem:regularised_CE1_weakly_continuous} 
Let \math{(\rho_t, J_t)_{t \in (0,T)}} be a weak solution to the continuity equation on \math{\frakG}. 
Then \math{t \mapsto \int_\frakG \varphi \dd \mu_t} is continuous for every \math{\varphi \in C(\frakG)}.
\end{lemma}

\begin{proof}
Fix	\math{\varphi \in C(\frakG)}, take a continuous extension to \math{\frakG_\ext}, and define \math{\varphi^\varepsilon} accordingly.
Proposition \ref{prop:function-reg}.\ref{it:fct-conv} then implies that \math{\varphi^\varepsilon} converges uniformly to \math{\varphi} on \math{\frakG} as \math{\varepsilon \searrow 0}. As a result, the function \math{t \mapsto \int_\frakG \varphi^\varepsilon \dd \mu_t} converges uniformly to $t \mapsto \int_\frakG \varphi \dd \mu_t$.
Since \math{\varphi^\varepsilon} belongs to 
\math{C(\frakG) \cap C^1(\overline \frakE)} by Proposition \ref{prop:function-reg}.\ref{it:fct-reg}, we conclude that the mapping \math{t \mapsto \int_\frakG \varphi \dd \mu_t} is continuous, being a uniform limit of continuous functions.
\end{proof}

By duality, we obtain a natural regularisation for measures.

\begin{definition}[Regularisation of measures]\label{def:regularise-functions}
For \math{\mu \in \calM(\frakG)} we define \math{\mu^\varepsilon \in \calM(\frakG_\ext)}
by 
\begin{equation}\label{eq:kernel_approx1}
	\int_{\frakG_\ext} \varphi \dd \mu^\varepsilon 
		:= \int_{\frakG} \varphi^\varepsilon \dd \mu.
\end{equation}
for all $\varphi \in C(\frakG_\ext)$.

Analogously, for $J \in \calM(\frakE)$ we define $J^\eps\in\calM(\frakE_{\rm ext})$ as follows: first we  extend $J$ to a measure on $\frakG$ giving no mass to $\frakG\setminus\frakE$.
Then we define $\tilde J_\eps \in \calM(\frakG_{\rm ext})$ by the formula above.
Finally, we define $J_\eps \in \calM(\frakE_{\rm ext})$ by restriction of $\tilde J_\eps$ to $\frakE_{\rm ext}$.
\end{definition}

It is readily checked that the right-hand side defines a positive linear functional on $C(\frakG_\ext)$, so that $\mu^\eps$ is indeed a well-defined measure.

\begin{proposition}\label{prop:measure-reg}
The following properties hold for any $\eps > 0$:
\begin{enumerate}[(i)]
\item \label{it:mass} 
Mass preservation: $\mu^\eps(\frakG_\ext) = \mu(\frakG)$ for any $\mu \in \calM(\frakG)$.
\item \label{it:regularisation} 
Regularising effect:
For any $\mu \in \calP(\frakG)$, the measure $\mu^\varepsilon$ is absolutely continuous with respect to $\lambda$ with density 
\begin{align*}
	\rho^\varepsilon(x) = 
	\left\{ \begin{array}{ll}
\displaystyle \frac{1}{2 \eps} \mu\big(e \cap I_e(x)\big) ,
 & \text{for $x$ on $e$ in $\frakE$, }\\
\displaystyle \frac{1}{2\eps} \bigg( \indicator_{\{d(x,w) \leq 2 \eps\}} \mu(\{w\}) + \sum_{e \in \sE: w \in e} \mu\big(e \cap I_e(x)\big) \bigg),
 & \text{for $x$ on $e^\ext_w$, $w \in \sV,$}
 \end{array} \right.
\end{align*}
where
\begin{equation*}
I_e(x) := \left(\frac{x - \eps}{a_e^\eps}
	\vee \Big(-\frac{\ell_e}{2}\Big), \frac{x + \eps}{a_e^\eps} \wedge \frac{\ell_e}{2}\right).
\end{equation*}
In particular, 
	$\rho^\varepsilon(x) \leq \frac{1}{2 \eps}$ for all $x \in \frakG_\ext$.
\item 
\label{it:energy} 
Kinetic energy bound: For $\mu \in \calP(\frakG)$ and $v \in L^2(\mu)$, 
define 
	$J = v \mu|_\frakE \in \calM(\frakE)$. 
Consider the regularised measures 
	$\mu^\eps \in \calP(\frakG_\ext)$ 
and 
	$J^\eps \in \calM(\frakE_\ext)$.  
Then we have $J^\eps = v^\eps \mu^\eps$ for some $v^\eps \in L^2(\mu^\eps)$ and
\begin{equation}\label{eq:approx_est_1}
	 \int_{\frakE_\ext} \abs{v^\varepsilon}^2 
		 		\dd \mu^\varepsilon
	 \leq \int_\frakE \abs{v}^2 \dd \mu.
\end{equation}

\item For any $\mu \in \calP(\frakG)$ we have weak convergence $\mu^\eps \rightharpoonup \mu$ in $\calP(\frakG_\ext)$ as $\eps \to 0$.\label{it:convergence}

\item\label{it:reg-CE} Let $(\mu_t, J_t)_{t \in (0,T)}$ be a weak solution to the to the continuity equation \eqref{eq:ce_weak1}. Then the regularised pair $(\mu_t^\varepsilon, J^\eps_t)_{t \in (0,T)}$ is a weak solution to a modified continuity equation on $\frakG_\ext$ in the following sense: 
\begin{quote}
For every absolutely continuous function $\varphi$ on $\frakG_\ext$, the function \math{t \mapsto \int_{\frakG_\ext} \varphi \dd \mu_t^\varepsilon} is absolutely continuous and for a.e.\ $t \in (0,T)$ we have
\begin{equation}	\label{eq:ce_weak_regularised1}
\frac{\D}{\D t} \int_{\frakG_\ext} \varphi \dd \mu_t^\varepsilon = \int_{\frakE_\ext} \alpha^\varepsilon\nabla \varphi \cdot \dd J_t^\varepsilon,
\end{equation}
with $\alpha^\eps$ as in Definition~\ref{def:regularised_curves1}.
\end{quote}
\end{enumerate}
\end{proposition}

In order to prove \eqref{it:energy}, we will make use of the so-called Benamou--Brenier functional (see, e.g., \cite[Section~5.3.1]{Sant} for corresponding results in the Euclidean setting).

Define $K_2 := \{ (a,b) \in \reals \times \reals: a+ b^2/2 \leq 0 \}$. 
By a slight abuse of notation,
$C(\frakG, K_2)$ 
(resp. $L^\infty(\frakG, K_2)$) 
denotes the set of all continuous (resp. bounded and measurable) functions \math{a,b: \frakG \to \reals} such that \math{a+ b^2/2 \leq 0}.

\begin{definition}
The \emph{Benamou--Brenier functional} $\calB_2: \calM(\frakG) \times \calM(\frakE) \to \reals \cup \{ +\infty\}$ is defined by
\begin{equation*}
\calB_2(\mu, J) := \sup_{ (a,b) \in C(\frakG, K_2)} \bigg\{ \int_\frakG a \dd \mu + \int_\frakE b \dd J \bigg\}.
\end{equation*}
\end{definition}

Some basic properties of this functional are collected in the following lemma.

\begin{lemma}\label{lem:BB}
The following statements hold:
\begin{enumerate}[(i)]
\item\label{it:lem:BB:1} For $y,z \in \R$ we have
\begin{equation}	\label{eq:it:lem:BB:1} 
\alpha(z,y):=\sup_{(a,b) \in K_2} \{ a z + b y \} = 	\begin{cases}
								\displaystyle	\frac{\abs{y}^2}{2z} 	& \text{if } z>0, \\
											0				& \text{if } z= 0 \text{ and } y = 0, \\
											+\infty			& \text{otherwise.}
								\end{cases}
\end{equation}

\item\label{it:lem:BB:2} For 
	$\mu \in \calM(\frakG)$ 
and 
	$J \in \calM(\frakE)$ 
we have
\begin{equation}\label{eq:it:lem:BB:2} 
	\calB_2(\mu, J) 
		= \sup_{(a,b) \in L^\infty(\frakG, K_2)} 
			\bigg\{ 	
				\int_\frakG a \dd \mu
					 + \int_\frakE b \dd J 
			\bigg\}.
\end{equation}

\item\label{it:lem:BB:3} The functional $\calB_2$ is convex and lower semicontinuous with respect to the topology of weak convergence on $ \calM(\frakG) \times \calM(\frakE)$.

\item\label{it:lem:BB:4} 
If 	
	$\mu \in \calM(\frakG)$ 
is nonnegative and 
	$J \in \calM(\frakE)$
satisfies 
	$J \ll \mu|_\frakE$
with $J = v \mu|_\frakE$, 
then we have	
\begin{equation}	\label{eq:it:lem:BB:4}
\calB_2(\mu, J) = \frac{1}{2} \int_\frakE \abs{v}^2 \dd \mu.
\end{equation}
Otherwise, we have $\calB_2(\mu, J) = +\infty.$
\end{enumerate}
\end{lemma}
\begin{proof}
\eqref{it:lem:BB:1}: see \cite[Lemma~5.17]{Sant}. 

\medskip
\eqref{it:lem:BB:2}: Clearly, 
the right-hand side of \eqref{eq:it:lem:BB:2} is bounded from below by $\calB_2(\mu, J)$. 
To prove the reverse inequality, let \math{a,b: \frakG \to \reals} be measurable functions satisfying \math{a+ b^2/2 \leq 0}.
By Lusin's theorem (see, e.g., \cite[Theorem~7.1.13]{measure}) there exist functions $a_\delta, b_\delta \in C(\frakG, K_2)$ satisfying
\begin{equation*}
\mu(\{a \neq a_\delta\}) \leq \frac{\delta}{2},\quad \sup \abs{a_\delta} \leq \sup \abs{a} \quad \text{and} \quad \vert J \vert(\{b \neq b_\delta\}) \leq \frac{\delta}{2}, \quad \sup \abs{b_\delta} \leq \sup \abs{b}. 
\end{equation*}
Define $\tilde a_\delta := \min \{ a_\delta, - \vert b_\delta \vert^2/2 \}$, so that the inequality $\tilde a_\delta + b_\delta^2/2 \leq 0$ is satisfied. 
Hence, the pair $(\tilde a_\delta, b_\delta)$ is admissible for the supremum on the right-hand side of \eqref{eq:it:lem:BB:2}.
Since $\int_\frakG \tilde a_\delta \dd \mu + \int_\frakE b_\delta \dd J$ converges to $\int_\frakG a \dd \mu + \int_\frakE b \dd J$ as $\delta \searrow 0$, we obtain \eqref{eq:it:lem:BB:2}.

\medskip
\eqref{it:lem:BB:3}: This follows from the definition of $\calB_2$ as a supremum of linear functionals.

\medskip
\eqref{it:lem:BB:4}: 
Let $\mu$ be nonnegative and $J \ll \mu$ with $J = v \mu$. Setting $v=0$ on $\frakG\setminus\frakE$,  
 \eqref{eq:it:lem:BB:2} and \eqref{eq:it:lem:BB:1} yield
\begin{equation*}
\calB_2(\mu, J) 
	= \sup_{(a,b) \in L^\infty(\frakG, K_2)} \Bigl\{ \int_\frakG a + bv \dd \mu \Bigr\}
	= \frac{1}{2} \int_\frakE \abs{v}^2 \dd \mu.
\end{equation*}

To prove the converse, 
suppose first that there exists a Borel set 
	$A \subseteq \frakG$ with $\mu(A) < 0$. 
Pick $a = - k \indicator_A$ and $b \equiv 0$ with $k \geq 0$, so that $\calB_2(\mu, J) \geq - k \mu(A)$. 
Since $k$ can be taken arbitrarily large, we infer that $\calB_2(\mu, J) = + \infty$.
Now suppose $\mu$ is non-negative,
but the signed measure $J$ is not absolutely continuous with respect to $\mu|_\frakE$, i.e., there exists a $\mu$-null set $A \subseteq \frakE$ such that $J(A) \neq 0$. 
For 
	$a = - \frac{k^2}{2} \indicator_A$ and $b = k \indicator_A$ 
with $k \in \reals$, we have 
	$\calB_2(\mu, J) \geq k J(A)$, 
which implies the result.
\end{proof}

\begin{proof}[Proof of Proposition~\ref{prop:measure-reg}]
\eqref{it:mass}: The claim follows readily from the definitions. 

\medskip

\eqref{it:regularisation}: 
For $\varphi \in C(\frakG_{\rm ext})$
we have
\begin{align*}
	\int_{\frakG_\ext} 
		\varphi(y) 
	\dd \mu^\varepsilon(y)
	& = \int_{\frakG} \varphi^\varepsilon(x) \dd \mu(x)
	 =
	\sum_{w \in \sV}
		\varphi^\varepsilon(w)
		\mu\big(\{ w \}\big)
	+
	\sum_{e \in \sE}
		\int_e \varphi^\varepsilon(x) \dd \mu(x).	
\end{align*}
For $w \in \sV$, we note that $\varphi^\eps(w)$ is obtained by averaging $\varphi$ on a subset of the auxiliary edge $e_w^\ext$:
\begin{align*}
	\varphi^\eps(w)
	= \frac{1}{2\eps}
		\int_{e_w^\ext}	
		\indicator_{\{d(w,y) \leq 2 \eps\}}
			\varphi(y)
		\dd y.
\end{align*}
For $e \in \sE$ we obtain, interchanging the order of integration,
\begin{align*}
	\int_e \varphi^\varepsilon(x) \dd \mu(x)
	& = \frac{1}{2\eps}
		\int_{(-\ell_e/2, \ell_e/2)}
		\bigg(
			\int_{\alpha^\varepsilon_e x - \varepsilon}
		 ^{\alpha^\varepsilon_e x+ \varepsilon} 
		 	\varphi(y) \dd y 
		\bigg)
		\dd \mu(x)
	\\& = \frac{1}{2\eps}
		\int_{(-\ell_e/2-2\eps, \ell_e/2+2\eps)}
		\varphi(y)\mu\big(I_\eps(y)\big)
		\dd y 
\end{align*}
Combining these three identities, the desired result follows.
\medskip

\eqref{it:energy}: 
Take bounded measurable functions \math{a,b: \frakG \to \reals} satisfying \math{a+ b^2/2 \leq 0},
extend them by $0$ to $\frakG_{\rm ext}$,
and define regularised functions 
	\math{a^\varepsilon,b^\varepsilon:\frakG\to\R} 
as done for \math{\varphi} in \eqref{eq:kernel_approx1_function1}.
By Jensen's inequality and the fact that the regularisation is linear and positivity-preserving, we obtain
\begin{equation*}
a^\varepsilon(x) + \frac{1}{2} \abs{b^\varepsilon(x)}^2 \leq \bigl( a + {\textstyle \frac{1}{2}}\vert b\vert^2 \bigr)^\varepsilon(x) \leq 0 		\qquad \forall x \in \frakG,
\end{equation*}
i.e., \math{a^\varepsilon} and \math{b^\varepsilon} are admissible for the supremum in \eqref{eq:it:lem:BB:2}. 
Therefore,
\begin{equation}\label{eq:regularised_L2_bound2}
	\int_{\frakG_\ext} a \dd \mu^\varepsilon + \int_{\frakE_\ext} b \dd J^\varepsilon 
	=
	\int_\frakG a^\varepsilon \dd \mu + \int_\frakE b^\varepsilon \dd J 
	\leq 
	\frac{1}{2} \int_\frakE \abs{v}^2 \dd \mu. 
\end{equation}
The result follows by taking the supremum over all admissible functions \math{a} and \math{b}.

\medskip
\eqref{it:convergence}: This follows from the uniform convergence of $\varphi^\varepsilon$ to $\varphi$; see Proposition~\ref{prop:function-reg}\eqref{it:fct-conv}.

\medskip
\eqref{it:reg-CE}: The function $\varphi$ is absolutely continuous, hence a.e.\ differentiable on $\frakE_\ext$. 
Proposition~\ref{prop:function-reg}\eqref{it:fct-reg} yields $\varphi^\varepsilon \in  C^1(\overline \frakE) \cap C(\frakG)$ and
\begin{equation}		\label{eq:regularised_derivative_identity1}
\nabla \varphi^\varepsilon(x) = \alpha^\varepsilon(x) (\nabla \varphi)^\varepsilon(x) 		\qquad \forall x \in \frakE.
\end{equation}
Since $\alpha^\varepsilon$ is constant on each metric edge in $\frakE$, we obtain, using the continuity equation and the definition of the regularisation,
\begin{align*}
\frac{\D}{\D t} \int_{\frakG_\ext} \varphi \dd \mu_t^\varepsilon 
	& = \frac{\D}{\D t} \int_{\frakG} \varphi^\varepsilon \dd \mu_t 
	= \int_{\frakE} \nabla \varphi^\varepsilon \cdot v_t \dd \mu_t 
	\\ & = \int_{\frakE} (\nabla \varphi)^\varepsilon \cdot \alpha^\varepsilon v_t \dd \mu_t 
	= \int_{\frakE_\ext} \nabla \varphi \cdot (\alpha^\varepsilon v_t^\varepsilon) \dd \mu_t^\varepsilon.
\end{align*}
\end{proof}

\bigskip
Now we are ready to prove the second part of Theorem~\ref{thm:continuity1}: we adapt the proof of \cite{gigli2015continuity}, where (much) more general metric measure spaces are treated, but stronger assumptions on the measures are imposed (namely, uniform bounds on the density with respect to the reference measure). 
Here we consider more general measures using the above regularisation procedure. 

\subsection{Conclusion of the characterisation of absolutely continuous curves}

In the proof of the second part of Theorem~\ref{thm:continuity1}, we make use of the Hopf--Lax formula in metric spaces and its relation to the dual problem of optimal transport.

\begin{definition}[Hopf--Lax formula]\label{def:Hopf--Lax}
For a real-valued function \math{f} on a geodesic Polish space \math{(X,d)}, we define \math{Q_t f: X \to \reals \cup \{-\infty\}} by
\begin{equation*}
Q_t f(x) := \inf_{y \in X} 
	\Big\{ 
		f(y) + \frac{1}{2t} d^2(x,y)
	\Big\}
\end{equation*}
for all \math{t>0}, and \math{Q_0 f := f}.
\end{definition}

The operators $(Q_t)_{t \geq 0}$ form a  semigroup of nonlinear operators with the following well-known properties; see \cite{AGS2013} for a systematic study.

\begin{proposition}[Hopf--Lax semigroup]\label{pro:Hopf--Lax1}
Let $(X, d)$ be a geodesic Polish space. For any Lipschitz function \math{f: X \to \reals} the following statements hold:
\begin{enumerate}[(i)]
\item \label{it:HL-Lip} 
For every \math{t\geq 0} we have \math{\Lip(Q_t f) \leq 2 \Lip(f)}.
\item \label{it:HL-HJ}
For every \math{x \in X}, the map \math{t \mapsto Q_t f(x)} is continuous on \math{\R_+}, locally semi-concave on \math{(0,\infty)}, and the inequality 
\begin{equation}\label{eq:Kuwadas_lemma}
\frac{\D}{\D t} Q_t f(x) + \frac{1}{2} \lip(Q_t f)^2(x) \leq 0
\end{equation}
holds for all \math{t \geq 0} up to a countable number of exceptions. 
\item \label{it:HL-usc}
The mapping \math{(t,x) \mapsto \lip(Q_t f)(x)} is upper semicontinuous on \math{(0,\infty) \times X}.
\end{enumerate}
\end{proposition}

\begin{proof}
	\emph{(i)}: This statement can be derived from 
	\cite[Proposition 3.4]{AGS2013}. 
	For the convenience of the reader we provide the complete argument here. 

	Fix $t > 0$ and $x \in X$. For $z \in X$ we  write $F(t,x,z) := f(z)  + \frac1{2t}d^2(x, z)$. 
	We claim that 
	\begin{align*}
		Q_t f(x) 
		= \inf_{\substack{z \in X \\ d(x, z) \leq 2 t L }}
			F(t,x,z),	
	\end{align*}
	where $L$ denotes the Lipschitz constant of $f$.
	Indeed, if $d(x,z) > 2 t L$, we have
	\begin{align*}
		F(t,x,z) 
		=
		f(z) + \frac1{2t}d^2(x, z)
		\geq f(x) - L d(x, z) + \frac1{2t}d^2(x, z)
		> f(x)
		= F(t,x,x), 
	\end{align*}
	which implies the claim.
	
	Fix now $y \in X$ and $\eps > 0$.
	Using the claim, we may pick $z \in X$ 
	such that 
		$d(y, z) \leq 2 t L$ and  
		$F(t,y,z) \leq Q_t f(y) + \eps$.
	Then: 
	\begin{align*}
		Q_t f(x) - Q_t f(y)
		& \leq F(t,x,z) - F(t,y,z) + \eps
		 = \frac1{2 t} \Big(  d^2(x,z) -  d^2(y,z) \Big) 
			+ \eps
		\\& \leq
			\frac{d(x,y)}{2 t}
			\Big(  d(x,z) +  d(y,z) \Big) 
			+ \eps
		 \leq
			\frac{d(x,y)}{2 t}
			\Big(  d(x,y) + 4 t L \Big) 
			+ \eps.
	\end{align*}	
	Reversing the roles of $x$ and $y$, this estimate readily yields 
	\begin{align*}
		\lip(Q_t f)(x) \leq 2 L.
	\end{align*}
	Since $X$ is assumed to be a geodesic space, we have $\Lip(Q_t f) = \sup_x \lip(Q_t f)(x)$ and the result follows.

	\smallskip
		
	\emph{(ii):} See \cite[Theorem 3.5]{AGS2013}.
	
	\smallskip

	\emph{(iii):} See \cite[Propositions 3.2 and 3.6]{AGS2013}.
\end{proof}

We can now conclude the proof of Theorem \ref{thm:continuity1} on the characterisation of absolutely continuous curves in the Wasserstein space over a metric graph.

\begin{proof}[Proof of \eqref{it:continuity1ii} in Theorem~\ref{thm:continuity1}]
Without loss of generality, we set $T=1$.
The main step of the proof is to show that 
\begin{equation}\label{eq:CE_estimate1}
	W_2^2(\mu_0, \mu_1) 
		\leq 
	\int_0^1 \norm{v_r}_{L^2(\mu_r)}^2 \dd r.
\end{equation}
From there, a simple reparametrisation argument (see also \cite[Lemma~1.1.4 \& 8.1.3]{AGS}) yields
\begin{equation*}
W_2^2( \mu_t, \mu_s) 
	\leq \frac{1}{\abs{s-t}} \int_s^t \norm{v_r}_{L^2(\mu_r)}^2 \dd r
\end{equation*}
for all $0 \leq s < t \leq 1$,
which implies the absolute continuity of the curve \math{(\mu_t)_{t \in (0,1)}} in \math{W_2(\frakG)} as well as the desired bound \math{\abs{\dot \mu}(t) \leq \norm{v_t}_{L^2(\mu_t)}} for every Lebesgue point \math{t \in (0,1)} of the map $t \mapsto \norm{v_t}_{L^2(\mu_t)}^2$.

\medskip 

Thus, we have to show \eqref{eq:CE_estimate1}.
To this aim, we will work on the supergraph \math{\frakG_\ext \supseteq \frakG}.

By Kantorovich duality (Proposition \ref{prop:Kantorovich-duality}), there exists
$\varphi \in C(\frakG)$ satisfying
\begin{equation}\label{eq:proof:dualformulation2}
	\frac{1}{2} W_2^2( \mu_0, \mu_1)
	 = \int_{\frakG} Q_1 \varphi \dd \mu_1 
	 	- \int_{\frakG} \varphi \dd \mu_0. 
\end{equation}
Moreover, $\varphi$ is Lipschitz, which follows from the fact that $\varphi$ is $c$-concave with $c(x,y)=\frac12\sd(x,y)^2$ and $(\frakG, \sd)$ is compact.
We consider Lipschitz continuous extensions of \math{\varphi} and \math{Q_1 \varphi} to \math{\frakG_\ext}, both constant on each auxiliary metric edge in $\overline \frakE_\ext$. In particular, \math{\varphi} and \math{Q_1 \varphi} are Lipschitz on \math{\frakG_\ext}.

Set \math{J_t := v_t \mu_t} and consider a regularised pair \math{(\mu_t^\varepsilon, J_t^\varepsilon)_{t \in (0,1)}} as defined by \eqref{eq:kernel_approx1}.
We write
\begin{equation}\begin{aligned}
	\label{eq:CE_estimate2}
\int_{\frakG_\ext} Q_1 \varphi \dd \mu_1^\varepsilon 
	- \int_{\frakG_\ext} \varphi \dd \mu_0^\varepsilon 
 & = 
 \sum_{i=0}^{n-1} 
 \bigg( 
	 	\int_{\frakG_\ext}
		 \big(Q_{(i+1)/n} \varphi 
 		- Q_{i/n} \varphi \big)
		 \dd \mu_{(i+1)/n}^\varepsilon 
 \\ & \qquad \qquad\qquad
 + \int_{\frakG_\ext} Q_{i/n} \varphi \dd\big(\mu_{(i+1)/n}^\varepsilon - \mu_{i/n}^\varepsilon\big) \bigg),
\end{aligned}\end{equation}
and bound the two terms on the right-hand side separately.

\medskip
\emph{Bound 1.} \ 
To estimate the first term on the right-hand side of \eqref{eq:CE_estimate2}, we use \eqref{eq:Kuwadas_lemma} to obtain
\begin{equation}\begin{aligned}
\label{eq:CE_estimate2.1}
	 \sum_{i=0}^{n-1} 
		\int_{\frakG_\ext} 
		\big(Q_{(i+1)/n} \varphi - Q_{i/n} \varphi \big)
		\dd \mu^\varepsilon_{(i+1)/n}
	& \leq - \frac{1}{2} \sum_{i=0}^{n-1} 
			\int_{\frakG_\ext} \int_{i/n}^{(i+1)/n} 
				\lip^2(Q_t \varphi) 
			\dd t \dd \mu^\varepsilon_{(i+1)/n} 
		\\ &
	= - \frac{1}{2} \int_{\frakG_\ext \times (0,1)} 
						 \lip^2(Q_t \varphi)(x) 
					 \dd{\pmb \mu}^\varepsilon_n(x,t),
\end{aligned}\end{equation}
where the measures \math{{\pmb \mu}^\varepsilon_n := \sum_{i=0}^{n-1} \mu^\varepsilon_{(i+1)/n} \otimes \calL^1\vert_{(i/n, (i+1)/n)}} are defined on \math{\frakG_\ext \times (0,1)}. 

To show weak convergence of the sequence \math{({\pmb \mu}^\eps_n)_{n \in \naturalnumbers}}, we take \math{\psi \in C(\frakG_\ext \times [0,1])}. 
Note that $t \mapsto \mu_t$ is weakly continuous by Lemma \ref{lem:regularised_CE1_weakly_continuous}, hence $t \mapsto \mu_t^\eps$ is weakly continuous as well. 
Consequently, $\int_{\frakG_\ext} \psi(\cdot, t) \dd \mu_{\lfloor tn \rfloor/n}^\eps \to \int_{\frakG_\ext} \psi(\cdot, t) \dd \mu_t^\eps$ for every $t$. 
Integrating in time over $(0,1)$, we infer, using dominated convergence, that \math{{\pmb \mu}^\varepsilon_n} converges weakly to \math{{\pmb \mu}^\varepsilon := \int_0^1 \mu^\varepsilon_t \otimes \delta_t \dd t 
} as \math{n \to \infty}.

As $\lip^2(Q_t \varphi)$ is not necessarily continuous, an additional argument is required to pass to the limit in \eqref{eq:CE_estimate2.1}. 
For this purpose, we observe that Proposition~\ref{prop:measure-reg}.\ref{it:regularisation} yields ${\pmb \mu}^\varepsilon_n \ll \lambda \otimes \calL^1$ with a density $\bfrho_n^\eps(x,t)\leq 1/(2\varepsilon)$ for $x \in \frakG_\ext$ and $t \in (0,1)$. 
In particular, the family $(\bfrho_n^\eps)_{n \in \naturalnumbers}$ is uniformly integrable with respect to \math{\lambda \otimes \calL^1}. 
Consequently, the Dunford--Pettis Theorem (see, e.g., \cite[Theorem~4.7.18]{measure}) implies that $(\bfrho_n^\eps)_{n \in \naturalnumbers}$ has weakly-compact closure in $L^1(\frakG_\ext \times (0,1))$. 
Since $\lip^2(Q_t \varphi)$ is bounded, we may pass to the limit $n \to \infty$ in \eqref{eq:CE_estimate2.1} and infer that 
\begin{equation}\begin{aligned}\label{eq:CE_estimate3}
	\limsup_{n \to \infty} 
		\sum_{i=0}^{n-1} 
		\int_{\frakG_\ext} 
		\big(Q_{(i+1)/n} \varphi - Q_{i/n} \varphi \big)
		\dd \mu^\varepsilon_{(i+1)/n} 	
	& \leq - \frac12 \int_{\frakG_\ext \times (0,1)} 
			 \lip^2(Q_t \varphi)(x)
		 \dd {\pmb \mu}^\varepsilon(x,t) 
\\& =
	- \frac12 \int_0^1\int_{\frakE_\ext} 
					 \lip^2(Q_t \varphi)
					 \dd \mu_t^\varepsilon \dd t,
\end{aligned}\end{equation}
where we use that \math{\mu_t^\varepsilon \ll \lambda} on \math{\frakG_\ext} to remove the set of nodes \math{\sV} from the domain of integration.

\medskip
\emph{Bound 2.} \ 
We now treat the second term in \eqref{eq:CE_estimate2}.
As $(\mu_t)_{t \in (0,1)}$ belongs to a weak solution to the continuity equation and we know from Proposition~\ref{prop:function-reg}.\eqref{it:fct-reg}) that 
\math{(Q_{i/n} \varphi)^\varepsilon} belongs to 
\math{C^1(\overline \frakE) \cap C(\frakG)}, 
we infer that the mapping \math{t \mapsto \int_{\frakG} (Q_{i/n} \varphi)^\varepsilon \dd \mu_t} is absolutely continuous. 
Therefore,
\begin{equation}\begin{aligned}\label{eq:CE_estimate4}
	& \sum_{i=0}^{n-1} 
		\int_{\frakG_\ext} Q_{i/n} 
			\varphi 
		\dd(\mu_{(i+1)/n}^\varepsilon - \mu_{i/n}^\varepsilon)
	 = \sum_{i=0}^{n-1} \int_{\frakG}
	 	 (Q_{i/n} \varphi)^\varepsilon
		 	 \dd(\mu_{(i+1)/n} - \mu_{i/n})
	\\& = 
	 	\sum_{i=0}^{n-1} \int_{i/n}^{(i+1)/n} \biggl( \int_{\frakE} \nabla(Q_{i/n} \varphi)^\varepsilon \dd J_t \biggr) \dd t 
	\\ & = 
	\sum_{i=0}^{n-1} \int_{i/n}^{(i+1)/n} 
			 \biggl( \sum_{e \in \sE}
			 	 \alpha^\varepsilon_e 
				 \int_e (\nabla Q_{i/n} \varphi)^\varepsilon
				 \dd J_t \biggr) \dd t 
	\\ & = 
	\sum_{i=0}^{n-1} \int_{i/n}^{(i+1)/n} 
			 \biggl( \sum_{e \in \sE}
			 	 \alpha^\varepsilon_e 
				 \int_e \nabla Q_{i/n} \varphi \cdot 
				 \dd J_t^\eps \biggr) \dd t 
	\\ & \leq \frac{\alpha^\varepsilon_{\max}}{2} 
		\sum_{i=0}^{n-1} 
			\int_{i/n}^{(i+1)/n}
			\int_{\frakE_\ext} 
							 \abs{\nabla Q_{i/n} \varphi}^2 
			\dd \mu_t^\varepsilon \dd t 
	+ \frac{\alpha^\varepsilon_{\max}}{2} 
			\int_0^1 \int_{\frakE_\ext} 
				\abs{ v_t^\varepsilon}^2 
			\dd \mu_t^\varepsilon \dd t,		
\end{aligned}\end{equation}
where \math{\alpha^\varepsilon_{\max} := \max_{e \in \sE} \alpha^\varepsilon_e} and \math{J^\eps_t=v^\eps_t\mu^\eps_t}.

By Proposition~\ref{pro:Hopf--Lax1}.\ref{it:HL-usc}, we have the bound \math{\limsup_{n \to \infty} \lip^2(Q_{\lfloor nt \rfloor / n} \varphi) \leq \lip^2(Q_t \varphi)}. 
As $\varphi$ is Lipschitz continuous, \eqref{it:HL-Lip} in the same proposition shows that $\sup_{t,x} \lip(Q_t \varphi)(x) < \infty$.
Thus, we may invoke Fatou's lemma to obtain
\begin{align*}
\limsup_{n \to \infty} 
	\sum_{i=0}^{n-1} 
		\int_{i/n}^{(i+1)/n} 
		\int_{\frakE_\ext} 
			\abs{\nabla Q_{i/n} \varphi}^2 
		\dd \mu_t^\varepsilon \dd t 
 & \leq 
	 \int_0^1 \int_{\frakE_\ext} 
	 	 \lip^2(Q_t \varphi) 
	\dd \mu_t^\eps \dd t.
\end{align*}
Using this estimate together with Proposition~\ref{prop:measure-reg}.\ref{it:energy}, we obtain
\begin{equation}\begin{aligned}
	\label{eq:CE_estimate5}
\limsup_{n \to \infty}
		\sum_{i=0}^{n-1} \int_{\frakG_\ext} 
			Q_{i/n} \varphi 
		\dd \big(\mu_{(i+1)/n}^\varepsilon 
				- \mu_{i/n}^\varepsilon\big) 	
	&	 \leq \frac{\alpha^\varepsilon_{\max}}{2} 
			\int_0^1 \int_{\frakG_\ext} 
				\lip^2(Q_t \varphi)
			 \dd \mu_t^\varepsilon \dd t 
	\\&	\qquad + \frac{\alpha^\varepsilon_{\max}}{2} 
			 \int_0^1 \int_{\frakE} 
			 	\abs{ v_t}^2
			 \dd \mu_t \dd t.
\end{aligned}\end{equation}

\medskip
\emph{Combination of both bounds.} \ 
Recalling \eqref{eq:CE_estimate2}, we use \eqref{eq:CE_estimate3} and \eqref{eq:CE_estimate5} to obtain
\begin{equation*}
\int_{\frakG_{\mathrlap \ext}} Q_1 \varphi \dd \mu_1^\varepsilon 
	- \int_{\frakG_{\mathrlap \ext}} \varphi \dd \mu_0^\varepsilon 
	 \leq \frac{\alpha^\varepsilon_{\max}}{2} \int_0^1 \int_{\frakE} \abs{ v_t}^2 \dd\mu_t \D t + 
	 \frac{\alpha^\varepsilon_{\max} - 1}{2} \int_{\frakG_\ext} \int_0^1 \lip^2(Q_t \varphi) \dd \mu_t^\varepsilon \dd t.
\end{equation*}
Using Proposition~\ref{prop:measure-reg}.\ref{it:convergence}, the fact that $\alpha^\varepsilon_{\max} \to 1$, and the bound $\sup_{t} \Lip(Q_t \varphi) < \infty$, we may pass to the limit $(\eps \to 0)$ to obtain
\begin{equation}\label{eq:CE_estimate7}
	\int_{\frakG} Q_1 \varphi \dd \mu_1 
		- \int_{\frakG} \varphi \dd \mu_0 
			\leq \frac{1}{2} 
				\int_0^1 \int_{\frakE} 
					\abs{ v_t}^2 
				\dd\mu_t \dd t. 
\end{equation}
In view of \eqref{eq:proof:dualformulation2}, this yields the result. 
\end{proof}

\begin{corollary}[Benamou--Brenier formula]
For any \math{\mu, \nu \in \calP(\frakG)}, we have
\begin{equation}
W_2^2(\mu, \nu) = \min 
	\bigg\{ \int_0^1 \int_\frakE \vert v_t \vert^2 
				\dd \mu_t \dd t \bigg\},
\end{equation}
where the minimum runs over all weak solutions to the continuity equation $(\mu_t, v_t\mu_t)_{t \in [0,1]}$ satisfying $\mu_0 = \mu$ and $\mu_1 = \nu$.
\end{corollary}

\begin{proof}
As \math{(\calP(\frakG), W_2)} is a geodesic space, we may write
\begin{equation*}
W_2^2(\mu,\nu) = \min\bigg\{ \int_0^1 \abs{\dot \mu}^2(t) \dd t \bigg\},
\end{equation*}
where the minimum runs over all absolutely continuous curves \math{(\mu_t)_{t \in [0,1]}} connecting \math{\mu} and \math{\nu}. 
Therefore, the result follows from Theorem~\ref{thm:continuity1}.
\end{proof}

\section{Lack of geodesic convexity of the entropy}
\label{sec:2}

In this section we consider the entropy functional
	\math{\Ent : \calP(\frakG) 
		\to (-\infty + \infty]} 
defined by
\begin{equation}		
\label{eq:ent_def_flat1}
	\Ent(\mu) := 
	\left\{ \begin{array}{ll}
	\displaystyle 
		\int_{\frakG} \rho \log \rho \dd \lambda \quad & \mbox{if } \mu = \rho \lambda, \\ 
	+\infty & \mbox{otherwise.} 
\end{array}\right.
\end{equation}
As is well known, this functional is lower semicontinuous on \math{(\calP(\frakG),W_2)}; see, e.g., 
\cite[Corollary 2.9]{Liero-Mielke-Savare:2018}.

A celebrated result by McCann asserts that $\Ent$ is geodesically convex on \math{(\calP(\R^d),W_2)}, the Wasserstein space over the Euclidean space $\R^d$.
More generally, on a Riemannian manifold $\calM$, the relative entropy (with respect to the volume measure) is geodesically $\kappa$-convex on \math{(\calP(\calM),W_2)} for $\kappa \in \R$, if and only if the Ricci curvature is bounded below by $\kappa$, everywhere on $\calM$ 
\cite{Otto-Villani:2000,Cordero-Erausquin-McCann-Schmuckenschlager:2001,Renesse-Sturm:2005}.

Metric graphs are prototypical examples in which such bounds fail to hold.
Here we present an explicit example, which shows that 
 the functional $\Ent$ on the metric space \math{(\calP(\frakG), W_2)} over a metric graph $\frakG$ induced by a graph with maximum degree larger than $2$ is \emph{not} geodesically $\kappa$-convex for any \math{\kappa \in \reals}.

\begin{example}\label{ex:three_leaves}
Consider a metric graph induced by a graph with 3 leaves as shown in Figure~\ref{fig:three_leaves}. 

\begin{figure}[ht]
	\includegraphics[width=0.28\textwidth]{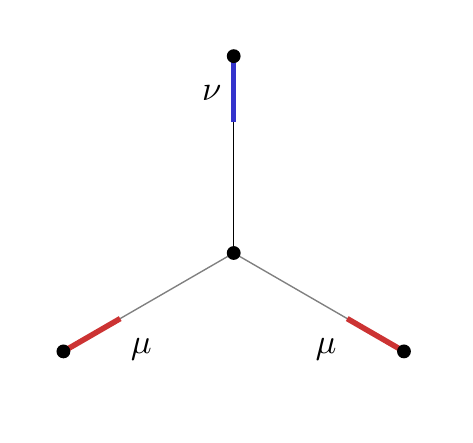}
	\caption{The support of probability measures \math{\mu} and \math{\nu} on a metric graph induced by a oriented star with 3 leaves}
\label{fig:three_leaves}
\end{figure}

We impose an edge weight $1$ on each of the edges $\math{$e_1, e_2, f}.
Consider the probability measures $\mu, \nu \in \calP(\frakG)$ with respective densities $\rho,\eta : \frakG \to \R_+$ given by
\begin{equation*}
	\rho(x) := \begin{cases}\frac{1}{2 \varepsilon} \indicator_{[0,\varepsilon]}(x), &x\in e_1 \text{ or } x\in e_2,\\ 
	0, & x\in f,\end{cases},	
		\quad 
	\eta(x) := \begin{cases}0, & x\in e_1 \text{ or } x\in e_2,\\ \frac{1}{\varepsilon} \indicator_{[1-\varepsilon,1]} (x), & x\in f. \end{cases}
\end{equation*}

\begin{lemma}
The unique optimal coupling of $\mu$ and $\nu$ is given by monotone rearrangement from each of the edges  \math{e_1} and \math{e_2} to \math{f}, i.e., by the map $T$ with
\begin{equation*}
T(x) = 1-\varepsilon +x \in f \qquad \text{ for } x\in e_1 \text{ or } x\in e_2.
\end{equation*}
\end{lemma}

\begin{proof}
Let $\pi$ be an optimal coupling of $\mu$ and $\nu$ and decompose it as $\pi=\pi_1+\pi_2$, where $\pi_i$, $i=1,2$, are couplings of $\mu_i$, the restriction of $\mu$ to $e_i$, and some $\nu_i$ a measure on $f$ such that $\nu_1+\nu_2=\nu$. Necessarily $\pi_1,\pi_2$ are also optimal. By standard optimal transport theory on the interval $e_i\cup f\subset \R$, $\pi_i$ is given by a map $T_i$, the monotone rearrangement from $\mu_i$ to $\nu_i$. If we show that $\nu_1=\nu_2$, then $T_1=T_2=T$ with $T$ as above and the claim is proven. To this end, assume by contradiction that $\nu_1\neq \nu_2$ and hence $T_1\neq T_2$. We consider the coupling $\pi'=\pi_1'+\pi_2'$ where $\pi_1'=\pi_2'=(\pi_1+\pi_2)/2$ (here, we identify $e_1\cup f$ and $e_2\cup f$ in the obvious way). Since $T_1\neq T_2$, the support of $\pi_1'$ is not contained in the graph of a function. Since $\pi'$ is optimal, the couplings $\pi_i$ are also optimal between their marginals $\mu_i$ and $\frac12\nu$, and thus have to be induced by the monotone rearrangement map $T$, a contradiction.  
\end{proof}

Consequently, the constant speed-geodesic from \math{\mu} to \math{\nu} is thus given by \math{({T_{t}}_{\#}\mu)_{t \in [0,1]}} where $T_t$ is the linear interpolation of $T$ and the identity on \math{e_1\cup f} and \math{e_2\cup f} respectively, more precisely
\begin{equation*}
T_t(x) := 	\begin{cases}
			x+(2 - \varepsilon)t \in e_i, & \text{if } x \leq 1 - (2-\varepsilon)t, \\
			x+(2 - \varepsilon)t -1 \in f, & \text{if } x > 1 - (2-\varepsilon)t,
		\end{cases}
\end{equation*}
for \math{x \in e_i}, \math{i \in \{1,2\}}. 

Set $t_0^\eps := \frac{1-\eps}{2-\eps}$ and $t_1^\eps := \frac{1}{2-\eps}$. The relative entropy of \math{\mu_t} is given by:
\begin{align*}
\Ent(\mu_t) =& 	\begin{cases}
			 \log \bigl(\tfrac{1}{2\varepsilon} \bigr), & 
			t \in [0, t_0^\eps], \\
			\frac{1}{\eps}\big(1 - (2 - \varepsilon)t\big)
				\log 
					\bigl(\tfrac{1}{2}
					\bigr)
			+\log\bigl(\frac{1}{\eps}\bigr), 
				& 	t \in [t_0^\eps, t_1^\eps], \\
			\log \bigr( \tfrac{1}{\varepsilon} \bigr), 		& \text{if }
			t \in [t_1^\eps, 1].
					\end{cases}
\end{align*}
Thus, $t \mapsto \Ent(\mu_t)$ is piecewise affine; see Figure \ref{fig:entropy_plot}.
It follows that
\math{t \mapsto \Ent(\mu_t)} is not \math{\kappa}-convex for any \math{\kappa \in \reals}.

\begin{figure}[ht]
	\includegraphics[width=0.5\textwidth]{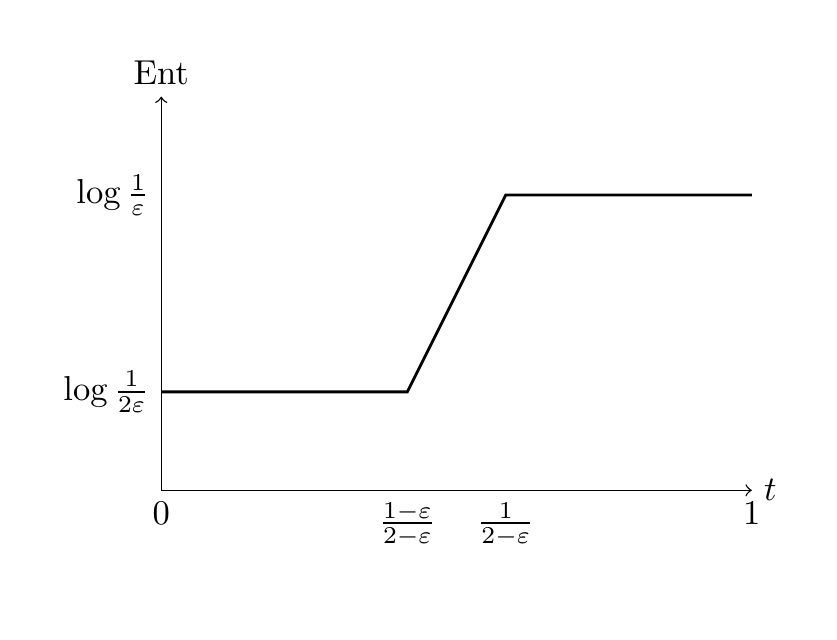}
	\caption{Plot of the entropy along the geodesic interpolation}
\label{fig:entropy_plot}
\end{figure}
\end{example}

\section{Gradient flows in the Wasserstein space over a metric graph}
\label{sec:3}

In this section we study gradient flows 
in the Wasserstein space over a metric graph. 
Namely, we consider diffusion equations 
on metric graphs arising as the gradient flow 
of free energy functionals composed as the sum of 
entropy, potential, and interaction energies. 
We give a variational characterisation of these diffusion equations via energy-dissipation identities 
and we discuss the approximation of solutions 
via the Jordan--Kinderlehrer--Otto scheme (minimizing movement scheme). 
This provides natural analogues on metric graphs of the corresponding classical results in Euclidean space. 
We follow the approach from the Euclidean case; 
see in particular \cite[Section 10.4]{AGS}, 
and adapt it to the current setting.
\medskip

Let  $V: \overline\frakE \to \reals$ be Lipschitz continuous and define the weighted volume measure $\mathfrak{m}:=e^{-V}\lambda$ on $\frakG$ (since $\lambda$ gives no mass to vertices, the potential ambiguity of $V$ there does not matter). We consider the following functionals on $\calP(\frakG)$:
 \begin{enumerate}[$(i)$]
 \item the \emph{relative entropy} $\calE_V:\calP(\frakG)\to(-\infty,\infty]$ defined by
 \begin{equation}\label{eq:def-F}
 \calE_V(\mu):={\rm Ent}[\mu|\mathfrak m]= \begin{cases} 
		\int_{\frakG}
			\rho(x)\log \rho(x)
		\dd\mathfrak m(x) & \text{ if }\mu=\rho\mathfrak m,\\
 +\infty & \text{ otherwise}. \end{cases}
 \end{equation}

 \item the \emph{interaction energy} $\calW:\calP(\frakG)\to \R$ defined by
 \[\calW(\mu) := 
 	\frac12
	 	\int_{\frakG \times \frakG} 
		 	W(x,y)
		\dd\mu(x) \dd\mu(y),
 \]
where $W:\frakG\times\frakG\to \R$ is symmetric and Lipschitz continuous.
\end{enumerate}
Moreover, we define $\calF:\calP(\frakG)\to(-\infty,+\infty]$ as the sum of the previous quantities:
\[\calF:=\calE_V+\calW.\]

Note that $\calE_V$ is bounded from below (by Jensen's inequality and finiteness of $\mathfrak m$) and lower semicontinuous with respect to~weak convergence.
Moreover, $\calW$ is bounded and continuous with respect to weak convergence.

Further note that for $\mu \in \calP(\frakG)$ with $\mu\ll \lambda$ we can write
\[\calE_V(\mu) = \calE(\mu)+\calV(\mu),\]
where $\calE(\mu)={\rm Ent[\mu|\lambda]}$ is the Boltzmann entropy on $\frakG$ and $\calV:\calP(\frakG)\to \R$ defined by
 \[
	 \calV(\mu) := \int_{\frakG} V(x)\dd \mu(x)
\]
is the \emph{potential energy}. The latter is well defined for $\mu\ll\lambda$ despite the potential discontinuity of $V$ at the vertices.

\subsection{Diffusion equation and energy dissipation}

We consider the following diffusion equation on the metric graph $\frakG$ given by 
\begin{equation}\label{eq:MKVeq}
	\partial_{t}\eta
	= \Delta\eta	
	+\nabla\cdot\big(\eta\big(\nabla V+\nabla W[\mu]\big)\big).
\end{equation}
Here $\mu=\eta\lambda$ is a probability on $\mathfrak G$ and we set $W[\mu](x) := \int_{\frakG} W(x,y)\dd\mu(y)$.
In analogy with the classical Euclidean setting we will show below that this PDE is the Wasserstein gradient flow equation of the free energy $\calF$.
Though the setting of metric graphs is one-dimensional, we prefer to use multi-dimensional notation such as $\Delta$ and $\nabla \cdot$ for the sake of clarity.

We consider the following notion of weak solution for \eqref{eq:MKVeq}.
\begin{definition}\label{def:MKV-weak}
We say that a curve $(\eta_{t})_{t\in[0,T]}$ of probability densities w.r.t.~$\lambda$ on $\frakG$ is a weak solution to \eqref{eq:MKVeq} if for $\rho_t:=\eta_t\cdot e^V$ we have $\rho_{t}\in W^{1,1}(\overline{\frakE})\cap C(\frakG)$  for a.e.~$t$, and the pair 
$(\mu,J)$ given by 
\begin{align*}
	\mu_{t}=\rho_{t}\mathfrak m \quad\text{and}\quad J_{t}
	= - \left(\nabla\rho_{t} +\rho_{t}\nabla W[\mu_{t}]\right)\mathfrak m
\end{align*}
is a weak solution to the continuity equation in the sense of Definition \ref{def:dist_et_weak_solution1}, i.e., we ask that $t\mapsto\mu_{t}$ is weakly continuous and for every 
$\varphi \in C^1_c\big((0,T) \times\overline\frakE\big)$
such that  
$\varphi_t,\partial_t\varphi_t
\in 
	C^1(\overline\frakE) \cap C({\frakG})$
for all $t \in (0,T)$, 
we have
\begin{equation}\label{eq:MKVweak_solution}
	\int_0^T\! \biggl( \int_{\frakG} \partial_t \varphi \rho_t \dd \mathfrak m  -
	\int_{\overline\frakE} \nabla \varphi\cdot \big(\nabla \rho_{t} +\rho_{t}\nabla W[\mu_{t}]\big)\dd \mathfrak m \biggr) \dd t = 0.
\end{equation}
\end{definition}

\begin{remark}
Let us briefly consider the special case where $V=W=0$. Then~\eqref{eq:MKVeq} is simply the heat equation on a metric graph, which has been introduced already in~\cite{Lum80}. 
It is known since~\cite{KraMugSik07,Mug07} that the Laplacian with natural vertex conditions (continuity across the vertices along with a Kirchhoff-type condition on the fluxes) is associated with a Dirichlet form, hence the Cauchy problem for this PDE is well-posed on $L^2(\frakG)$: 
more precisely, it is governed by an ultracontractive, Markovian $C_0$-semigroup $(\e^{t\Delta})_{t\ge 0}$ that extrapolates to $L^p(\frakG)$ for all $p\ge 1$, 
as well as to $C(\frakG)$ and, 
by duality, to the space $\mathcal M(\frakG)$ of Radon measures on $\frakG$. 
For any initial value $\mu_0\in {\mathcal M}(\frakG)$, $(t,x)\mapsto \rho(t,x)=(\e^{t\Delta}\mu_0)(x)$ defines a classical solution to the Cauchy problem for the heat equation with initial value $\mu_0$.
It is easy to see that this solution is a weak solution in the sense of Definition~\ref{def:MKV-weak} as well.
\end{remark}

The dissipation of the free energy along solutions to \eqref{eq:MKVeq} at $\mu=\rho\mathfrak m$ is formally given by
\begin{align*}
\frac{\D}{\D t}\calF(\mu) = - \int_{\frakE} \Big|\frac{\nabla \rho}{\rho}+\nabla W[\mu]\Big|^{2}  \dd \mu.
\end{align*}
This motivates the following definition.
\begin{definition}[Energy dissipation functional]
	\label{def:dissipation}
The \emph{energy dissipation functional} 
	$\calI : \calP(\frakG) \to [0, + \infty]$ 
is defined as follows. 
If $\mu = \rho \mathfrak m$ with 
	$\rho\in W^{1,1}(\overline\frakE)\cap C(\frakG)$ 
and $\nabla \rho+\rho\nabla W[\mu] 
	= \boldsymbol{w}\rho$ 
for some $\boldsymbol{w}\in L^{2}(\mu)$ we set 
\begin{equation*}
\calI(\mu) := \int_{\frakE} \vert \boldsymbol{w} \vert^2 \dd\mu.
\end{equation*}
Otherwise, we set $\calI(\mu)=+\infty$.
\end{definition}

\begin{remark}
	We emphasize that continuity of $\rho$ on $\frakG$ is required for finiteness of $\calI(\mu)$ in this definition. It is not sufficient that $\rho$ belongs to $W^{1,1}(\overline{\frakE})$. 
	The continuity is important in the proof of Theorem \ref{thm:EDE} below, as it ensures spatial continuity for gradient flow curves (i.e.,~curves of maximal slope with respect to the upper gradient $\sqrt{\calI}$) and it allows us to identify them with weak solutions to the diffusion equation \eqref{eq:MKVeq}. The requirement of spatial continuity for weak solutions to \eqref{eq:MKVeq}  in Definition \ref{def:MKV-weak} is essential in order to couple the dynamics on different edges across common vertices.
\end{remark}

We collect the following properties of the dissipation functional.

\begin{lemma}\label{lem:dissipation-lsc-convex}
Let $\{\mu_{n}\}_n\subseteq\calP(\frakG)$ be a sequence weakly converging to $\mu\in \calP(\frakG)$ such that 
\begin{equation}\label{eq:energy-diss-bdd}
\sup_{n} \mathcal F(\mu_{n})<\infty \quad \text{and} \quad
\sup_{n} \calI(\mu_{n})<\infty.
\end{equation}
Then we have 
\begin{equation}
\mathcal I(\mu)\leq \liminf_{n}\calI(\mu_{n}).
\end{equation}
\end{lemma}

\begin{proof}
First note that we can rewrite $\calI(\mu)$ as the integral functional 
\begin{equation}\label{eq:int-funct}
\calI(\mu)=\mathcal G_{\alpha}(\mu,J)=\int_{\frakE}\alpha\big(\frac{\D \mu}{\D\sigma},\frac{\D J}{\D\sigma}\big)\D\sigma,
\end{equation}
where $\alpha:[0,\infty)\times \R\to[0,\infty]$ is the lower semicontinuous and convex function defined by
\begin{equation}\label{eq:alpha}
\alpha(s,u)=\begin{cases}
|u|^{2}/s & \hbox{if }s>0,\\
0 & \hbox{if }s=0  \hbox{ and }u=0,\\
+\infty & \text{otherwise},
\end{cases}
\end{equation}
and $J=\pmb w \mu= (\nabla \rho+\rho\nabla W[\mu])\mathfrak m$. Here $\sigma$ is any measure such that $\mu,J\ll\sigma$; 
its choice is irrelevant by $1$-homogeneity of $\alpha$.

Now, let $\mu_{n}=\rho_{n}\mathfrak m$. 
Lower semicontinuity of $\calF$ and \eqref{eq:energy-diss-bdd} imply that $\calF(\mu)<\infty$. 
Therefore we can write $\mu = \rho\mathfrak m$ for a suitable density $\rho$. 
Superlinearity of $r\mapsto r\log r$ 
implies that $\rho_{n}$ 
converges weakly to $\rho$ in $L^{1}(\frakG,\mathfrak m)$. 

Recall that $\calI(\mu_{n})=\int |\pmb w_{n}|^{2}\D\mu_{n}$ with $U_{n}:=\rho_{n}\pmb w_{n}=\nabla \rho_{n}+\rho_{n}\nabla W[\mu_{n}]$. 
H\"older's inequality and the bound \eqref{eq:energy-diss-bdd} yield that the measures $J_{n}=U_{n}\mathfrak m$ have uniformly bounded total variation. Hence up to extracting a subsequence, we have that $J_{n}$ converges weakly* to a measure $J$ on $\frakE$. Lower semicontinuity of the integral functional $\mathcal G_{\alpha}$ yields 
\begin{equation}\label{eq:lscG}
\mathcal G_{\alpha}(\mu,J)\leq \liminf_{n}\calI(\mu_{n})<\infty.
\end{equation}
This allows us to write $J=\pmb w\rho\mathfrak m$ for a $\pmb w\in L^{2}(\mu)$ and $\mathcal G_{\alpha}(\mu,J)=\int |\pmb w|^{2}\D\mu$. 
Since $\rho_n\in W^{1,1}(\overline\frakE)\cap C(\frakG)$, we have for any function  $s\in C_c^1(\frakE)$ by integration by parts,
\begin{equation}
	\begin{aligned}
		\label{eq:IBP}
		-\int_\frakG 
		\rho_n \nabla s \dd\lambda 
		&=
		\int_\frakG s \nabla\rho_n \dd \lambda 
		\\& = \int_\frakG s e^V \big(\nabla\rho_n +\rho_n\nabla W[\mu_n]\big)  \dd\mathfrak m - \int_{\frakG}s e^V \nabla W[\mu_n]\dd\mu_n\\
		&= \int_\frakG s e^V\dd J_n - \int_\frakG s e^V \nabla W[\mu_n]\dd\mu_n.
	\end{aligned}
\end{equation}
The convergence of $J_{n}$ to $J$ weakly and $\rho_n$ to $\rho$ in $L^1$ together with the fact that $J= \pmb w \rho\mathfrak m$ gives no mass to $\sf V$ and the boundedness of  $\nabla W$ allows us to pass to the limit and obtain
\begin{equation*}
	-\int_\frakG  \rho\nabla s
	\dd\lambda 
	= \int_\frakG  s e^V \dd J - \int_\frakG s e^V\nabla W[\mu]\dd\mu 
	= \int_\frakG  s   \pmb w \rho \dd\lambda  - \int_\frakG s \rho\nabla W[\mu]\dd\lambda 
\end{equation*}
for all $s \in C_c^1(\frakE)$.
We infer that $\rho\in W^{1,1}(\overline\frakE)$ and that  $\rho\pmb w=\nabla \rho +\rho\nabla W[\mu]$. 
In particular, $\rho\in C(\overline\frakE)$ by the Sobolev embedding theorem. 

Let us now show that $\rho\in C(\frakG)$.
For this purpose, note that each pair of adjacent edges $(e, f)$ 
can be identified with the interval $[-\ell_e,\ell_f]$. 
Consider $s\in C^1(\frakE)$ such that $s=0$ for all $g\neq e,f$ and $s=r$ on $e,f$ for some $r\in C_c^1(-\ell_e,\ell_f)$. 
Repeating the argument above, we infer that $\rho\in W^{1,1}([-\ell_e,\ell_f])$ and in particular that $\rho$ is continuous at $0$. 
We thus obtain that $\rho\in C(\frakG)$.
Hence from \eqref{eq:lscG} we obtain the claim.
\end{proof}

Let us denote by $\mathcal I_0$ the energy dissipation functional with $W=0$, more precisely, if $\mu=\rho\mathfrak m$ with $\rho\in W^{1,1}(\overline\frakE)\cap C(\frakG)$ with $\nabla\rho=\boldsymbol{w}\rho$ for some $\boldsymbol{w}\in L^2(\mu)$ we set 
\[\mathcal I_0(\mu) := \int_\frakE |\boldsymbol{w}|^2\dd\mu.\]
Otherwise, we set $\mathcal I_0(\mu)=+\infty$. Similarly as for $\calI$ we can write $\calI_0(\mu)=\int\alpha(\rho,\nabla\rho)\dd\mathfrak m$ with $\mu=\rho\mathfrak m$ where $\alpha$ is the function in \eqref{eq:alpha}. Then $\mathcal I_0$ is a convex and lower semi-continuous functional by the previous lemma. We note that under the assumption that $W$ is Lipschitz, $\mathcal I(\mu)$ is finite if and only if $\mathcal I_0(\mu)$ is finite.

Next, we observe that finiteness of the $\mathcal I_0$ implies a quantitative $L^\infty$ bound on the density.

\begin{lemma}\label{lem:Fisher-Linfty}
 For $\mu\in \calP(\frakG)$ with $\calI_0(\mu)<\infty$ and $\mu=\rho \mathfrak m$ we have $\rho\in C(\frakG)$ with 
 \[\|\rho\|_\infty \leq A\sqrt{\calI_0(\mu)},\]
 where the constant $A>0$ depends on $\|V\|_\infty$.
\end{lemma}

\begin{proof}
The continuity of $\rho$ follows from the definition of $\calI_0(\mu)$. 
Using H\"older's we obtain with $\nabla\rho=\rho\boldsymbol{w}$,
\begin{align*}
\int |\nabla \rho| \dd\lambda &\leq e^{\|V\|_\infty} \int |\nabla \rho| \dd\mathfrak m \leq e^{\|V\|_\infty}\Big(\int |\boldsymbol{w}|^2\dd\mu\Big)^\frac12.
\end{align*}
The claim then follows immediately from the Sobolev embedding theorem applied to each of the finitely many edges. 
\end{proof}

\subsection{Energy-dissipation equality}

The main result of this section is the following gradient flow characterisation of the diffusion equation \eqref{eq:MKVeq}.

\begin{theorem}\label{thm:EDE}
	For any 2-absolutely continuous curve $(\mu_t)_{t\in[0,T]}$ in $(\calP(\frakG),W_2)$ with $\mathcal F(\mu_0)<\infty$ we have 
	\begin{equation*}
		\mathcal L_{T}(\mu) 
		:= \mathcal F(\mu_T) -\mathcal F(\mu_{0})+ \frac{1}{2} \int_0^T \abs{\dot \mu}^2(r) + \calI(\mu_r) \dd r \geq 0 .
		\end{equation*}
Moreover, we have $\mathcal L_{T}(\mu)=0$ if and only if
		$\mu_t=\rho_{t} \mathfrak m$ is a weak solution to \eqref{eq:MKVeq} 		in the sense of Definition~\ref{def:MKV-weak}.
	\end{theorem}
	
The main step in proving this result is to establish a chain rule for the free energy $\mathcal F$ along absolutely continuous curves in $(\calP(\frakG), W_2)$. 
Recall the definition of $\alpha$ from \eqref{eq:alpha}.

\begin{proposition}
	[Chain rule]
	\label{prop:chain-rule}
 Let $(\mu_t)_{t\in[0,T]}$ be a $2$-absolutely continuous curve in $(\calP(\frakG), W_2)$ satisfying
 	$\int_0^T\calI(\mu_t)\dd t < +\infty$.
Write $\mu_t = \rho_t \mathfrak m$ and let $J_t=U_t\mathfrak m$ be an optimal family of momentum vector fields, i.e., $(\mu_t, J_t)_{t \in [0,T]}$ solves the continuity equation and 
$|\dot\mu|^{2}(t) = 
	\int
		\alpha( \rho_{t}, U_{t} )
	\D\mathfrak m$.
Let $\rho_t\boldsymbol{w}_{t}=\nabla\rho_{t}+\rho_{t}\nabla W[\mu_t]$ as in Definition \ref{def:dissipation}. Then, 
	$t \mapsto \calF(\mu_t)$
is absolutely continuous and we have
 \begin{align}\label{eq:chain-rule}
 \frac{\dd}{\dd t}
		\calF(\mu_t) 
		= \int_\frakE \langle\boldsymbol{w_{t}},\D J_t\rangle \quad \text{ for a.e.~}t\in[0,T].
 \end{align}
\end{proposition}

\begin{proof}
We first note that the assumptions ensure that also $\int_0^T\calI_0(\mu_t)\dd t<\infty$.
Hence, we have
\begin{equation}\label{eq:finiteIA}
\int_0^T\int_\frakE \alpha(\rho_t,\nabla\rho_t)\dd\mathfrak m\dd t <\infty,\qquad \int_0^T\int_{\frakE}\alpha(\rho_t,U_t)\dd\mathfrak m \dd t<\infty.
\end{equation}

We proceed by a twofold regularisation. 
Using a family \math{(\eta^\delta)_{\delta>0}} of even and smooth approximation kernels with compact support in $[-\delta, \delta]$, we  regularise in time via
\begin{equation*}
\rho^{\delta}_t := \int_{-\delta}^\delta \eta^\delta(s) \rho_{t-s} \dd s,
 \end{equation*}
 and set  $\mu^{\delta}_{t}=\rho^{\delta}_{t}\mathfrak m$. Here we extend $\rho$ to a curve on the time-interval 
 	$[-\delta, T+\delta]$ 
which is constant on $[-\delta,0]$ and $[T,T+\delta]$. 
Similarly defining $J_{t}^{\delta}$ as the time-regularisation of $J_{t}$ we obtain that $(\mu^\delta,J^\delta)$ is a solution to the continuity equation. Convexity of $\calI_0$ and the Benamou--Brenier functional yield that $\int_0^T\calI_0(\mu^\delta_t)\dd t$ and $\int_0^T|\dot\mu^\delta|^2(t)\dd t<\infty$. Hence, \eqref{eq:finiteIA} also holds with $\rho^\delta_t, U^\delta_t$ in place of $\rho_t,U_t$. 
Moreover, we must have $U^\delta=0$ on $\{\rho^\delta=0\}$.

Further, we regularise the logarithm and define for $\eps>0$ the function $F_\eps:[0,\infty)\to \R$ by setting
\[F_\eps(r)= (r+\eps)\log(r+\eps)-\eps\log\eps.\]
Then we define the regularized free energy $\calF_\eps$ of $\mu=\rho\mathfrak m$ via
\[\calF_\eps(\mu) = \int_{\frakG} F_\eps(\rho)\dd\mathfrak m + \frac12\int_{\frakG\times\frakG} W (x,y)\dd\mu(x)\dd\mu(y).\]
Let us set $g^{\eps,\delta}=F'_\eps(\rho^{\delta})=1+\log(\rho^\delta+\eps)$. Note that by Lemma \ref{lem:Fisher-Linfty}, $\rho_\delta$ is bounded and thus also $g^{\eps,\delta}$ is bounded. We will first show that
\begin{align}
\calF_\eps(\mu^\delta_T)-\calF_\eps(\mu^\delta_0) 
\label{eq:chain-reg1}
&= \int_0^T\int_{\frakG}\langle\nabla g^{\eps,\delta}_t+\nabla W[\mu^\delta_t],\dd J^\delta_t\rangle\dd t.
\end{align}
Then passing to the limit $\delta,\eps\to 0$ will yield the claim. 

While establishing \eqref{eq:chain-reg1} we write $g$ instead of $g^{\eps,\delta}$ for simplicity. We have 
\begin{align*}\nonumber
\calF_\eps(\mu^\delta_T)-\calF_\eps(\mu^\delta_0) 
&= \int_0^T\int_{\frakG}\big(F'_\eps(\rho^\delta) + W[\mu^\delta_t]\big)\partial_t\rho^\delta_t\dd\mathfrak m\dd t.
\end{align*}
Differentiation under the integral sign is justified by boundedness of $F'_\eps(\rho^\delta)$ and $W$ and the regularity in time of $\rho^\delta$. Note that $W[\mu^\delta_t]$ is an admissible test function in the continuity equation for $(\mu^\delta,J^\delta)$. Thus, to establish \eqref{eq:chain-reg1} it remains to show that the continuity equation can also be used on the function $g=F'_\eps(\rho^\delta)$.
Recall that $g$ is bounded with $\nabla g= \nabla \rho^\delta/(\rho^\delta+\eps)$. We can approximate $g$ uniformly by functions $g^\alpha\in C^1(\overline\frakE\times[0,T])\cap C(\frakG\times[0,T])$ as follows. First extend $g$ to $\frakG_{\rm ext}\times[0,T]$ with constant values on the additional edge incident to vertex $v$ equal to the value of $g$ at $v$. Then we apply the regularising procedure \eqref{eq:kernel_approx1_function1}. The continuity equation for $(\mu^\delta,J^\delta)$ yields
\begin{align*}
\int_0^T\int_{\frakG} g^\alpha_t\partial_t\rho^\delta_t\dd\mathfrak m\dd t= 
\int_0^T\int_{\frakE}\langle\nabla g^\alpha_t,\dd J^\delta_t\rangle.
\end{align*}
Passing to the limit as $\alpha\to0$ then will finish the proof of \eqref{eq:chain-reg1}.
Convergence of the left hand side is immediate since $g$ is bounded and so $g^\alpha$ is bounded uniformly in $\alpha$. For the right hand side we estimate with $J^\delta= U^\delta\mathfrak m$:
\begin{align*}
\bigg|\int_0^T\int_\frakE\langle\nabla g^\alpha_t-\nabla g_t,\dd J^\delta_t\rangle\dd t\bigg| 
	& \leq \bigg(\int_0^T\int_{\frakE}|\nabla g^\alpha_t-\nabla g_t|^2\rho^\delta_t\dd\mathfrak m\dd t\bigg)^\frac12
\\ & \qquad \times	
	\bigg(\int_0^T\int_\frakE\alpha(\rho^\delta_t,U^\delta_t)\dd\mathfrak m \dd t\bigg)^\frac12 .
\end{align*}
The second factor is finite as noted above. To estimate the first factor, we recall that $\rho^\delta$ is bounded. Hence for a suitable constant $C < \infty$ 
\begin{align*}
\int_0^T\int_{\frakE}|\nabla g^\alpha_t-\nabla g_t|^2\rho^\delta_t\dd\mathfrak m\dd t 
\leq
C \int_0^T\int_{\frakE}|\nabla g^\alpha_t-\nabla g_t|^2\dd\lambda\dd t 
\end{align*}
Using \eqref{eq:regularised_diff_quotient1} and dominated convergence, the latter term goes to zero as $\alpha \to0$ if we show that 
\[\int_0^T\int_{\frakE}|\nabla g|^2\dd\lambda\dd t<\infty.\]
But using $\nabla g = \nabla \rho^\delta/(\rho^\delta+\eps)$  we can estimate
\begin{align*}
\int_0^T\int_{\frakE}|\nabla g|^2\dd\lambda\dd t
&\leq 
e^{\|V\|_\infty} \int_0^T\calI_0(\mu\delta_t)\dd t<\infty. 
\end{align*}
Thus, \eqref{eq:chain-reg1} is established.
\medskip

 We will now pass to the limits $\delta,\eps\to 0$ in \eqref{eq:chain-reg1}, starting with the right-hand side.
 
 For the limit $\delta\to0$, note that $U^{\delta}\to U$ a.e.~and 
  $\pmb w^{\eps,\delta}\to \pmb w^\eps$ a.e.~as $\delta\to0$ with $\pmb{w}^\eps=\nabla\rho/(\rho+\eps) +\nabla W[\mu]$. Dominated convergence then yields that as $\delta\to0$:
\begin{align*}
\int_{0}^{T}\int_{\frakE}\langle\pmb w^{\eps,\delta}_{t}, U^{\delta}_{t}\rangle\dd\mathfrak m\dd t \rightarrow \int_{0}^{T}\int_{\frakE}\langle\pmb w^\eps_{t}, U_{t}\rangle\dd\mathfrak m\dd t . 
\end{align*}
Indeed,  we have the majorant
\begin{align*}
|\langle\pmb{w}^{\eps,\delta},U^\delta\rangle| &\leq \frac12 |\pmb{w}^{\eps,\delta}|^2(\rho^\delta+\eps) +\frac12\frac{|U^\delta|^2}{\rho^\delta+\eps}
\leq
\alpha\big(\rho^\delta,\nabla\rho^\delta\big)+ C(\rho^\delta+\eps)+\frac12\alpha(\rho^\delta,U^\delta)\\
&\leq
\big[\alpha\big(\rho,\nabla\rho\big)\big]^\delta+ C(\rho^\delta+\eps) +\frac12\big[\alpha(\rho,U)\big]^\delta
\end{align*}
for a suitable constant $C$, using the fact $\nabla W[\rho^\delta]$ is uniformly bounded. Here, $\big[\cdot\big]^\delta$ denotes the time-regularisation of the function in brackets and we have used Jensen's inequality in the last step to interchange the regularisation operation with the convex function $\alpha$.
Since by assumption $\alpha\big(\rho,\nabla\rho\big)$ and $\alpha(\rho,U)$ are integrable on $[0,T]\times\frakE$, the majorant above indeed converges in $L^1\big([0,T]\times\frakE\big)$.

To further pass to the limit $\eps\to0$, we note that $\pmb{w}^\eps\to\pmb{w}$ a.e.~on the set $\{\rho>0\}$. Since $U=0$ a.e.~on the set $\{\rho=0\}$, we conclude that $\langle \pmb{w}^\eps,U\rangle\to\langle\pmb{w},U\rangle$ a.e. Using similar as before the majorant  $|\langle\pmb{w}^{\eps},U\rangle| \leq
\alpha\big(\rho,\nabla\rho\big)+ C \rho +\alpha(\rho,U)/2$, we conclude by dominated convergence.
\medskip

It remains to pass to the limit on the left-hand side in \eqref{eq:chain-reg1}. 
The weak continuity of $t\mapsto \mu_t$ implies that $\mu^\delta_t\rightharpoonup\mu_t$ weakly as $\delta\to0$. This is sufficient to conclude that $\mathcal W(\mu^\delta_t)\to\mathcal W(\mu_t)$. 
Note that for any $\mu=\rho\mathfrak m\in\mathcal P(\frakG)$ we have
\[\int F_\eps(\rho)\dd\mathfrak m = \mathcal E_V(\mu^\eps)-M\eps\log\eps,\]
with $\mu^\eps=(1+M)^{-1}(\mu+\eps\mathfrak m)$ and $M=\mathfrak m(\frakG)$.
Convexity of $r\mapsto r\log r$ and Jensen's inequality yield that $\mathcal E_V(\mu^{\delta,\eps}_t)\leq \mathcal E_V(\mu^\eps_t)$. 
Thus, lower semicontinuity of $\mathcal E_V$ under weak convergence shows that $\mathcal E_V(\mu^{\delta,\eps}_t)\to\mathcal E_V(\mu^\eps_t)$ and hence $\int F_\eps(\rho^\delta_t)\dd\mathfrak m\to \int F_\eps(\rho_t)\dd\mathfrak m$ as $\delta\to0$. 
The limit $\eps\to0$ is then easily achieved by monotone convergence.
From the convergence of the right-hand side of \eqref{eq:chain-reg1} and the assumption that $\calF(\mu_0)<\infty$ we finally conclude that $\calF(\mu_t)<\infty$ for all $t>0$ and that 
\begin{equation*}
 \calF(\mu_T) - \calF(\mu_0) = \int_0^T \int\langle\pmb w_t,\dd J_t\rangle \dd t. 
\end{equation*}
Hence $t\mapsto\calF(\mu_t)$ is absolutely continuous and \eqref{eq:chain-rule} follows.
\end{proof}

We can now prove Theorem \ref{thm:EDE}.

\begin{proof}[Proof of Theorem \ref{thm:EDE}]
Note that the right-hand side of \eqref{eq:chain-rule} may be estimated by means of H\"older's and Young's inequality as 
\begin{equation}\label{eq:chain-young}
\int_\frakE \langle\boldsymbol{w}_{t},U_t\rangle \dd\mathfrak m \geq - \sqrt{\int_{\frakE}\abs{\boldsymbol{w}_{t}}
^{2}\rho_{t}\dd\mathfrak m} \sqrt{\int_{\frakE}\frac{\abs{U_t}^2}{\rho_{t}} \dd\mathfrak m}\geq -\frac{1}{2} \int_\frakE \abs{\boldsymbol{w}_{t}}^2\rho_{t} \dd\mathfrak m -\frac12\int_{\frakE} \frac{\abs{U_t}^2}{\rho_{t}} \dd\mathfrak m.
\end{equation}
Hence, by integrating both sides of \eqref{eq:chain-rule} from $0$ to $T$ we obtain that $\mathcal L_{T}(\mu)\geq 0$. Moreover, we have equality if and only if for a.e.~$t$ and $\mu_{t}$-a.e.~we have $U_t=-\rho_{t}\boldsymbol{w}_{t}$. Now the continuity equation with $J_{t}=U_{t}\mathfrak m = -\rho_{t}\boldsymbol{w}_{t}\mathfrak m =-\nabla \rho_{t} +\rho_{t}\nabla W[\mu_{t}]\mathfrak m $ becomes the weak formulation of \eqref{eq:MKVeq}
\end{proof}

\subsection{Metric gradient flows}

Here, we recast the variational characterisation of Mc\-Kean--Vlasov equations on metric graphs from the previous section in the language of the theory of gradient flows in metric spaces. Let us briefly recall the basic objects. For a detailed account we refer the reader to \cite{AGS}.

Let $(X,d)$ be a complete metric space and let
$E:X\to(-\infty,\infty]$ be a function with proper domain, i.e., the
set $D(E):=\{x:E(x)<\infty\}$ is non-empty.

The following notion plays the role of the modulus of the gradient in
a metric setting.

\begin{definition}[Strong upper gradient]\label{def:upper-grad}
 A function $g:X\to[0,\infty]$ is called a \emph{strong upper
 gradient} of $E$ if for any $x\in AC\big((a,b);(X,d)\big)$ the
 function $g\circ x$ is Borel and
 \begin{align*}
 |E(x_s)-E(x_t)|~\leq~\int_s^tg(x_r)|\dot x|(r)\dd r \quad\forall~a\leq s\leq t\leq
 b\ .
 \end{align*}
\end{definition}

Note that by the definition of strong upper gradient, and Young's
inequality $ab\leq \frac12(a^2+b^2)$, we have that for all $s\leq t$:
\begin{align*}
 E(x_t) - E(x_s) +\frac12 \int_s^t g(x_r)^2 + |\dot x|^2(r)\dd r\geq 0.
\end{align*}

\begin{definition}[Curve of maximal slope]\label{def:curve-max-slope}
 A locally $2$-absolutely continuous curve $(x_t)_{t\in(0,\infty)}$ is
  called a curve of maximal slope of $E$ with respect to its strong upper
 gradient $g$ if $t\mapsto E(x_t)$ is non-increasing and
 \begin{align}\label{eq:cms}
 E(x_t) - E(x_s) +\frac12 \int_s^t g(x_r)^2 + |\dot x|^2(r)\dd r \leq 0 \quad\forall~0< s\leq t.
 \end{align}
 We say that a curve of maximal slope starts from $x_0\in X$ if $\lim_{t\searrow 0}x_t=x_0$.
\end{definition}

Equivalently, we can require equality in \eqref{eq:cms}. If a strong
upper gradient $g$ of $E$ is fixed we also call a curve of maximal
slope of $E$ (relative to $g$) a \emph{gradient flow curve}.

Finally, we define the (descending) metric slope of $E$ as the
function $|\partial E|:D(E)\to[0,\infty]$ given by
\begin{align}\label{eq:metric-slope-def}
 |\partial E| (x) = \limsup_{y\to x}\frac{\max\{E(x)-E(y),0\}}{d(x,y)}.
\end{align}

The metric slope is in general only a weak upper gradient for $E$, see
\cite[Theorem 1.2.5]{AGS}.  In our application to metric graphs we will show that the square root of the
dissipation $\calI$ provides a strong upper gradient for the free energy $\calF$.

\begin{corollary}\label{cor:gf-metric}
 The functional $\sqrt{\calI}$ is a strong upper gradient of $\calF$ on $(\calP(\frakG),W_{2})$. The curves of maximal slope for $\calF$ with respect to~this strong upper gradient coincide with weak solutions to \eqref{eq:MKVeq} satisfying $\int_{0}^{T}\calI(\mu_{t})\dd t<\infty$.
\end{corollary}

\begin{proof}
 For a $2$-absolutely continuous curve with $(\mu_{t})_{t}$ with $\int_{0}^{T}\calI(\mu_{t})\dd t<\infty$, the chain rule \eqref{eq:chain-rule} together with the estimate \eqref{eq:chain-young} yields
 \begin{equation}\label{eq:strong_upper_gradient1}
\abs{\calF(\mu_t) - \calF(\mu_s) } \leq \int_s^t \sqrt{\calI(\mu_r)} \abs{\dot \mu_r} \dd r \qquad \forall s, t \in [0,T]: s \leq t,
\end{equation}
i.e., $\sqrt{\calI}$ is a strong upper gradient. Theorem \ref{thm:EDE} yields the identification of curves of maximal slope.
\end{proof}

The dissipation functional $\calI(\mu)$ can be related to the metric slope of the free energy $\calF$ under suitable conditions on $\mu \in \calP(\frakG)$.

\begin{lemma}\label{lem:slope_Fisher1}
For any \math{\mu \in \calP(\frakG)} we have
		\begin{equation*}
			\calI(\mu) \leq \abs{\partial \calF}^2(\mu).
		\end{equation*}
\end{lemma}

\begin{proof}
We assume that 
	$\mu = \rho \mathfrak m=\tilde\rho\lambda \in \calP(\frakG)$ 
satisfies 
	$\abs{\partial \mathcal F}(\mu) < \infty$, 
since otherwise there is nothing to prove. 

\smallskip
\emph{Step 1.} We show first that 
	$\rho\in W^{1,1}(\overline\frakE)$ 
and $\nabla \rho +\rho\nabla W[\mu] = \rho \pmb w$ with $\pmb w \in L^2(\mu)$.

For this purpose, 
take any $\pmb s \in C^\infty(\overline{\frakE})$ which vanishes in a neighbourhood of every node and put $\tilde{ \pmb s}=e^V\pmb s$.
As a consequence, the mapping $r_t : \frakG \to \frakG$ defined by $r_t:= \Id + t \tilde{\pmb s}$ maps each edge into itself for $t > 0$ sufficiently small. 
We claim that:
	\begin{equation}\label{eq:directional_derivative1}
	\lim_{t \searrow 0} 
		\frac{\mathcal F\big(
				(\pmb r_t)_\# \mu\big) 
			- \mathcal F(\mu)}{t} 
		= \int_\frakG \big[-\rho\nabla\pmb s + \pmb s \rho\nabla  W[\mu]\big]\dd\lambda.
	\end{equation}
To show this, note that \math{\pmb r_t} is injective for every \math{t} small enough. Therefore, the change of variables formula yields that the density $\tilde \rho_t$ of $(\pmb r_t)_\# \mu$ with respect to $\lambda$ satisfies
\begin{align*}
	\tilde \rho(x) = \tilde \rho_t(\pmb r_t(x)) \nabla\pmb r_t(x).
\end{align*}
 Consequently, with $F(r)=r\log r$ we have
\begin{align*}
\mathcal F((\pmb r_t)_\# \mu) - \mathcal F(\mu)
	 = & \int_\frakG F\Bigl( \frac{\tilde\rho}{ \nabla\pmb r_t}\Bigr)\nabla\pmb r_{t} - F(\tilde\rho) \dd\lambda 
	 	+ \int_\frakG [V\circ \pmb r_t - V]
		 	\dd \mu\\
		&+\int_{\frakG\times\frakG}[W\circ (\pmb r_{t}\otimes\pmb r_{t})-W ]\dd\mu\otimes\mu.	
\end{align*}
Note that $\nabla\pmb r_{t}=1 +t\nabla\tilde{\pmb s}$. Dividing by $t$ and letting $t\searrow0$ and noting that $\nabla \tilde{\pmb s} = e^V(\nabla \pmb s+\pmb s\nabla V)$ and $\tilde\rho\tilde{\pmb s}=\rho\pmb s$ we deduce 
\eqref{eq:directional_derivative1}. 	

For $t > 0$ sufficiently small, we then have
\begin{equation*}
	W_2\big(\mu, (\pmb r_t)_\# \mu \big) 
		\leq t \norm{\tilde{\pmb s}}_{L^2(\mu)}.
\end{equation*}
This estimate, together with \eqref{eq:directional_derivative1}, implies 
\begin{equation*}
	 \int_\frakG \big[-\rho\nabla\pmb s + \pmb s \rho\nabla  W[\mu]\big]\dd\lambda
	= 	\lim_{t \searrow 0} 
	\frac{\mathcal F\big(
			(\pmb r_t)_\# \mu\big) 
		- \mathcal F(\mu)}{t} 
	\leq 
	\abs{\partial \mathcal F}(\mu) 
		\norm{\tilde{\pmb s}}_{L^2(\mu)}.
\end{equation*}
Hence, the left-hand side defines an $L^2(\mu)$-bounded linear functional on a subspace of $L^2(\mu)$. 
Using the Hahn--Banach theorem we may extend this functional to $L^2(\mu)$.
The Riesz representation theorem then yields a unique element \math{\boldsymbol{w} \in L^2(\mu)} such that \math{\norm{\boldsymbol{w}}_{L^2(\mu)} \leq \abs{\partial \mathcal F}(\mu)} and
\begin{equation}\label{eq:test-s}
 	\int_\frakG 
 		\big[
			- \rho\nabla\pmb s 
 			+ \pmb s \rho\nabla (V + W[\mu])
		\big]
	\dd\lambda
= 	\int_\frakG 
		\pmb w \tilde{\pmb s} 
	\dd \mu = \int_\frakG \pmb w \rho s \dd\lambda
\end{equation}
for all $\pmb s$ as above.

Considering in particular functions $\pmb s$ supported on a single edge, 
we infer that 
	$\rho\in W^{1,1}(\overline\frakE)$
and
 	$\nabla \rho +\rho\nabla W[\mu] = \rho\pmb w$ with $\pmb w \in L^2(\mu)$.
In particular, we have $\rho\in C(\overline\frakE)$ by Sobolev embedding. 

\smallskip

\emph{Step 2.} Next we show that $\rho$ belongs to $C(\frakG)$. 
For this purpose we repeat the argument above, for a different class of functions $\pmb s$.

Consider a pair of adjacent edges $e,f \in \sE$ with common vertex $v$. 
For the moment, we identify the concatenation of the two edges with the interval $[-\ell_e,\ell_f]$ such that $v$ corresponds to $0$.
Let $s : [-\ell_e, \ell_f] \to \R$ be a $C^\infty$-function vanishing in a neighbourhood of $-\ell_e$ and $\ell_f$. 
In particular, $r_t := \Id + t \tilde s$ maps $[-\ell_e, \ell_f]$ into itself for $t > 0$ sufficiently small, where $\tilde s=e^V s$. 
Let the maps $\pmb s : \frakE \to \R$ and $\bfr_t : \frakE \to \frakG$ be defined on $e\cup f$ by $s$ and $r_t$ through the identification with $[-\ell_e,\ell_f]$ and by setting $\pmb s=0$ and $\bfr_t={\rm id}$ on all other edges.
Repeating the argument from Step 1, we infer that
$\rho\in W^{1,1}([-\ell_e,\ell_f])$ with the identification of $e,
f$ with $[-\ell_e,\ell_f]$ as above. In particular, $\rho$ is continuous at $0$. Since the choice of the pair $e,f$ is arbitrary, we conclude that $\rho\in C(\frakG)$.

\smallskip
Combining both steps, we infer that $\calI(\mu) < \infty$ and
\begin{align*}
	\calI(\mu) 
	= \| \pmb w \|_{L^2(\mu)}^2
	\leq 
		\abs{\partial \calF}^2(\mu).
\end{align*}
\end{proof}

From Lemma \ref{lem:slope_Fisher1} and the lower semicontinuity of $\mathcal I$ we immediately infer that $\sqrt{\calI}$ lies below the relaxed metric slope, i.e., the lower semicontinuous relaxation of $|\partial \mathcal F|$ given by
\[|\partial^{-}\mathcal F| (\mu):= \inf\big\{\liminf_{n}|\partial \mathcal F|(\mu_{n})~:~\mu_{n}\rightharpoonup \mu\big\}.\]

\begin{corollary}\label{cor:rel-slope}
For all $\mu\in \calP(\frakG)$ we have $\sqrt{\calI(\mu)}\leq |\partial^{-}\mathcal F|(\mu)$.
\end{corollary}

\subsection{Approximation via the JKO scheme}
In this section, we consider the time-discrete variational approximation scheme of Jordan--Kinderlehrer--Otto for the gradient flow \cite{jordan1998variational}. \medskip

 Given a time step $\tau>0$ and an initial
datum $\mu_0\in \calP(\frakG)$ with $\mathcal F(\mu_0)<\infty$, we consider a
sequence $(\mu^\tau_n)_n$ in $\calP(\frakG)$ defined recursively via

\begin{align}\label{eq:min-move}
 \mu^\tau_0 = \mu_0,\quad \mu^\tau_n \in \underset{\nu}{\argmin} \Big\{\mathcal F(\nu)
 + \frac{1}{2\tau} W_{2}(\nu,\mu^\tau_{n-1})^2\Big\}.
\end{align}

Then we build a discrete gradient flow trajectory as the piecewise
constant interpolation $(\bar\mu^\tau_t)_{t\geq 0}$ given by
\begin{align}\label{eq:interpolation}
 \bar\mu^\tau_0 = \mu_0, \quad \bar\mu^\tau_t = \mu^\tau_n\ \text{if } t\in\big((n-1)\tau,n\tau]. 
\end{align}

Then we have the following result.

\begin{theorem}\label{thm:JKO}
 For any $\tau>0$ and $\mu_0\in \calP(\frakG)$ with $\mathcal F(\mu_0)<\infty$
 the variational scheme \eqref{eq:min-move} admits a solution
 $(\mu^\tau_n)_n$. As $\tau\to0$, for any family of discrete
 solutions there exists a sequence $\tau_k\to0$ and a locally
 $2$-absolutely continuous curve $(\mu_t)_{t\geq0}$ such that
 \begin{align}\label{eq:scheme-limit}
 \bar\mu^{\tau_k}_t \rightharpoonup \mu_t\quad \forall
 t\in[0,\infty).
 \end{align}
 Moreover, any such limit curve is a gradient flow curve of $\mathcal F$,
 i.e., a weak solution to the diffusion equation \eqref{eq:MKVeq}.
\end{theorem}

\begin{proof}[Proof of Theorem \ref{thm:JKO}]
The result basically follows from general results for metric gradient flows, where the scheme is known as the minimizing movement scheme; see \cite[Section 2.3]{AGS}.
 We consider the metric space $(\calP(\frakG),W_2)$ and endow
 it with the weak topology $\sigma$. It follows that \cite[Assumptions 2.1 (a,b,c)]{AGS} are satisfied. Existence of a solution to the variational scheme
 \eqref{eq:min-move} and of a subsequential limit curve $(\mu_t)_t$
 now follows from \cite[Corollary 2.2.2, Proposition 2.2.3]{AGS}. Moreover,
 \cite[Theorem 2.3.2]{AGS} gives that the limit curve is a curve of
maximal slope for the strong upper gradient $|\partial^-\mathcal F|$, i.e.,
 \begin{align*}
 \frac12\int_0^t|\dot\mu|^2(r) + |\partial^-\mathcal F(\mu_r)|^2\dd r + \mathcal F(\mu_t)
 \leq \mathcal F(\mu_0).
 \end{align*}
 Thus, by Corollary \ref{cor:rel-slope}, it is also a curve of maximal
 slope for the strong upper gradient $\sqrt{\calI}$. Theorem \ref{thm:EDE} yields the identification with weak solutions to \eqref{eq:MKVeq}.
 \end{proof}

 \subsection*{Acknowledgement} {\small 
ME acknowledges funding by the Deutsche Forschungsgemeinschaft (DFG), Grant SFB 1283/2 2021 – 317210226.
 DF and JM were supported by the European Research Council (ERC) under the European Union's Horizon 2020 research and innovation programme (grant agreement No 716117). JM also acknowledges support by the Austrian Science Fund (FWF), Project SFB F65. The work of DM was partially supported by the Deutsche Forschungsgemeinschaft (DFG), Grant 397230547. This article is based upon work from COST Action 18232 MAT-DYN-NET, supported by COST (European Cooperation in Science and Technology), www.cost.eu.
We wish to thank Martin Burger and Jan-Frederik Pietschmann for useful discussions.
We are grateful to the anonymous referees for their careful reading and useful suggestions.
 }

\bibliographystyle{my_plain}
\bibliography{bibdb}

\end{document}